\documentclass{amsart}
\usepackage{amsfonts, color}
\usepackage{latexsym}
\usepackage{amssymb}
\usepackage{amsmath}

\newcommand{\R}{\mathbb R}

\newcommand{\Z}{\mathbb Z}

\newcommand{\E}{\mathbb E}
\newcommand{\Pro}{\mathbb P}

\newtheorem{thm}{Theorem}[section]

\newtheorem{lemma}{Lemma}[section]

\theoremstyle{remark}
\newtheorem*{rmk}{Remark}

\numberwithin{equation}{section}

\begin{document}


\title{Sections of convex bodies in John's and minimal surface area position}

\author[D.\,Alonso]{David Alonso-Guti\'errez}
\address{\'Area de an\'alisis matem\'atico, Departamento de matem\'aticas, Facultad de Ciencias, Universidad de Zaragoza, Pedro Cerbuna 12, 50009 Zaragoza (Spain), IUMA}
\email[(David Alonso)]{alonsod@unizar.es}

\author[S. Brazitikos]{Silouanos Brazitikos}
\address{Department of Mathematics, National and Kapodistrian University of Athens, Panepistimioupolis 157-84,
Athens, Greece.}
\email[(Silouanos Brazitikos)]{silouanb@math.uoa.gr}
\thanks{The first named author is partially supported by MICINN Project PID-105979-GB-I00 and DGA Project E48\_20R. The second named author
is supported by the Hellenic Foundation for Research and Innovation (Project Number: 1849).}
\begin{abstract}
We prove several estimates for the volume, mean width, and the value of the Wills functional of sections of convex bodies in John's position, as well as for their polar bodies. These estimates extend some well-known results for convex bodies in John's position to the case of lower-dimensional sections, which had mainly been studied for the cube and the regular simplex. Some estimates for centrally symmetric convex bodies in minimal surface area position are also obtained.
\end{abstract}

\date{\today}
\maketitle
\section{Introduction and notation}\label{sec:IntroNotation}

Ball showed in \cite{B89} that if $B_\infty^n$ denotes the $n$-dimensional cube and $F\in G_{n,k}$ is a $k$-dimensional linear subspace of $\R^n$ then $|B_\infty^n\cap F|\leqslant 2^{\frac{n+k}{2}}$. He also obtained the bound
\begin{equation}\label{eq:SlicesCube}
|B_\infty^n\cap F|^{1/k}\leqslant\sqrt{\frac{n}{k}}|B_\infty^k|^{1/k},
\end{equation}
which is optimal if $k\mid n$.

It follows from results of Ball \cite{B92} that the $k$-dimensional
sections of a regular simplex with largest volume are exactly its $k$-dimensional faces. Let us also mention that
Webb showed in \cite{We96} that if $S_n$ denotes the regular simplex with inradius $r(S_n)=1$ then, for every hyperplane through the origin $F\in G_{n,n-1}$,
\begin{equation}\label{eq:SlicesSimplex}
|S_n\cap F|^\frac{1}{n-1}\leqslant\frac{1}{(\sqrt{2n(n+1)})^\frac{1}{n-1}}\sqrt{\frac{n(n+1)}{n-1}}|S_{n-1}|^\frac{1}{n-1}.
\end{equation}
There is equality for the sections passing through the origin that contain $n-1$ of the vertices.

Dirksen proved in \cite[Theorem 6.1]{D16}, the following estimate for the volume of $k$-dimensional sections of the regular simplex through the center of mass. If $S_n$ denotes the centered regular $n$-dimensional simplex with inradius $r(S_n)=1$ then for every $F\in G_{n,k}$ we have that
\begin{equation}\label{eq:kDimenionalSlicesSimplex}
|S_n\cap F|^{1/k}\leqslant\frac{1}{(k+1)^\frac{n-k}{2k(n+1)}}\sqrt{\frac{n(n+1)}{k(k+1)}}|S_k|^{1/k}.
\end{equation}

Besides, this estimate is asymptotically sharp.

Passing to the general case, we observe that it is not possible to obtain an upper bound for the volume of sections of a general centered convex body $K$ (i.e., a compact convex set with non-empty interior) without any additional assumption since, considering different positions of $K$ (i.e., affine images of $K$), we can obtain sections with volume as large as desired. For that matter we consider convex bodies in a particular position. In this paper we consider convex bodies in John's position and in minimal surface area position.

A convex body $K\subseteq\R^n$  is said to be in John's position if the Euclidean unit ball $B_2^n$ is contained in $K$ and for every non-degenerate linear map $T\in GL(n)$ such that $T(B_2^n)\subseteq K$ we have that $|T(B_2^n)|\leqslant|B_2^n|$, where $|\cdot|$ denotes the volume of a convex body. In other words, $K$ is in John's position if the Euclidean unit ball is the maximal volume ellipsoid contained in $K$. This position is uniquely determined up to orthogonal transformations. Both the cube and the regular simplex with inradius $1$ considered above are in John's position.

In \cite{B91}, Ball proved that among all convex bodies in John's position, $S_n$, the regular simplex with inradius $r(S_n)=1$, has the largest volume and, among all the centrally symmetric convex bodies in John's position, the cube $B_\infty^n$ has the largest volume. In a recent article \cite{Ma20}, Markessinis claimed to have obtained an upper bound for the volume of $k$-dimensional central sections of convex bodies in John's position. However, although the estimate given for central sections of centrally symmetric convex bodies in John's position is correct, the proof in the not necessarily symmetric case is not correct. In the following theorem we give an upper bound for the volume of central (and non-central) $k$-dimensional sections of an arbitrary convex body which is in John's position.

\begin{thm}\label{thm:VolumeJohnSections}
Let $K\subseteq\R^n$ be a convex body in John's position and $F\in G_{n,k}$. Then
$$
|K\cap F|^{1/k}\leqslant\frac{1}{(k+1)^{\frac{n-k}{2k(n+1)}}}\sqrt{\frac{n}{k}\frac{(n+1)}{(k+1)}}|S_k|^{1/k}.
$$
Furthermore, if $K$ is centrally symmetric
$$
|K\cap F|^{1/k}\leqslant\sqrt{\frac{n}{k}}|B_\infty^k|^{1/k}.
$$
Moreover, if $F$ is a $k$-dimensional affine subspace at distance $d$ from the origin and $K$ is a convex body in John's position then
$$
|K\cap F|^{1/k}\leqslant\sqrt{\frac{n(n+1)^{1+\frac{1}{k}}}{k(k+1)^{1+\frac{1}{k}}}}\left(\frac{n}{n+d^2}\right)^{\frac{1}{2k}}|S_k|^{1/k}.
$$
\end{thm}

\begin{rmk}
The proof of the symmetric case is the same as the one given by Markessinis. Nevertheless, we will reproduce it for the sake of completeness. We can also obtain it as a direct consequence of Theorem \ref{thm:WillsJohnSymmetric}, as well as a consequence of Theorem \ref{thm:Wills2JohnSections} below (see Section \ref{subsec:Wills}). This estimate is a sharp generalization of Ball's estimate \eqref{eq:SlicesCube} for the cube. Moreover, the case $k=1$ gives one more proof of John's theorem in the symmetric case: if
$K$ is a centrally symmetric convex body in $\R^n$ whose maximal volume ellipsoid is $B_2^n$ then $K\subseteq\sqrt{n}B_2^n$.
\end{rmk}

\begin{rmk}
Notice that we recover the estimate in \eqref{eq:kDimenionalSlicesSimplex}, which is asymptotically sharp for the simplex. Besides, if we take non-central sections by $k$-dimensional subspaces a distance $d=\sqrt{\frac{n(n-k)}{(k+1)}}$ from the origin, which is the distance from the origin to any $k$-dimensional face of $S_n$ we obtain that
$$
|K\cap F|^{1/k}\leqslant\sqrt{\frac{n(n+1)}{k(k+1)}}|S_k|^{1/k},
$$
which is exactly the volume of the $k$-dimensional faces of $S_n$. The estimate for general affine subspaces can also be obtained as a direct consequence of Theorem \ref{thm:Wills2JohnSections} below.
\end{rmk}

Dual to this position is the so called L\"owner's position. A convex body is said to be in L\"owner's position if the minimal volume ellipsoid containing it is the Euclidean unit ball. A convex body $K\subseteq\R^n$ is in John's position if and only if $K^\circ$ is in L\"owner's position, where $K^\circ$ denotes the polar body of $K$, defined as $K^\circ=\{x\in\R^n\,:\,\langle x,y\rangle\leqslant 1\,,\,\forall\,y\in K\}$. In \cite{B89}, Ball also showed that among all convex bodies in L\"owner's position, $\tilde{S}_n$, the regular simplex with circumradius $R(\tilde{S}_n)=1$, has the smallest volume and among all  centrally symmetric convex bodies in L\"owner's position,  $B_1^n$, the $\ell_1^n$-ball has the smallest volume.

Concerning the volume of polar bodies of sections of convex bodies in John's position (i.e, projections of convex bodies in L\"owner's position), it was proved by Barthe in his PhD thesis (see also \cite{A10}) that in the case of the $\ell_p^n$-balls, if $1\leqslant p\leqslant 2$ and $F\in G_{n,k}$ then
$$
|P_F(B_p^n)|^{1/k}\geqslant\left(\frac{k}{n}\right)^{\frac{1}{p}-\frac{1}{2}}|B_p^k|^{1/k}.
$$
In particular, we have the following estimate for the projections of $B_1^n$: For every $F\in G_{n,k}$
$$
|P_F(B_1^n)|^{1/k}\geqslant\sqrt{\frac{k}{n}}|B_1^k|^{1/k}.
$$
We can also obtain a lower bound for the volume of  $k$-dimensional projections of
convex bodies in L\"owner's position.

\begin{thm}\label{thm:VolumeLownerProjections}
Let $K\subseteq\R^n$ be a convex body in L\"owner's position and $F\in G_{n,k}$. Then
$$
|P_F(K)|^{1/k}\geqslant\sqrt{\frac{k}{n}}|\tilde{S}_k|^{1/k}
$$
Furthermore, if $K$ is centrally symmetric
$$
|P_F(K)|^{1/k}\geqslant\sqrt{\frac{k}{n}}|B_1^k|^{1/k}.
$$
\end{thm}

The mean width of a convex body $K\subseteq\R^n$ is defined as
$$
w(K)=\int_{S^{n-1}}h_K(\theta)d\sigma(\theta),
$$
where, for every $x\in\R^n$,  $h_K(x):=\sup\{\langle x,y\rangle\,:\,y\in K\}$ is the support function of $K$ at $x$ and $d\sigma$ denotes the uniform probability measure on $S^{n-1}$. In \cite{ScS95}, the authors proved that among all centrally symmetric convex bodies in John's position in $\R^n$, $w(K)$ is maximized if $K=B_\infty^n$. The not necessarily symmetric case was treated in \cite{Ba98}, where it was proved that among all convex bodies in John's position in $\R^n$, $w(K)$ is maximized if $K=S_n$, where $S_n$ denotes the regular simplex in John's position.  If we pass to them mean width of sections, then a direct consequence of \cite[Theorem 10]{BGMN2005}, is that for any $k$-dimensional linear subspace $F\in G_{n,k}$
$$
w(B_\infty^n\cap F)\leqslant \sqrt{\frac{n}{k}}w(B_\infty^k)
$$
and this estimate is sharp when $k\mid n$.

Furthermore, it was proved in \cite{ScS95} that among all centrally symmetric convex bodies in L\"owner's position in $\R^n$, $w(K)$ is minimized if $K=B_1^n$ and it was proved in \cite{Sc99} that among all convex bodies in L\"owner's position in $\R^n$, $w(K)$ is minimized if $K=\tilde{S}_n$. We will prove the following results on the mean width of sections of convex bodies in John's position

\begin{thm}\label{thm:MeanWidthJohnSetions}
Let $K\subseteq\R^n$ be a convex body in John's position and $F\in G_{n,k}$. Then
$$
w(K\cap F)\leqslant C\sqrt{\frac{n\log n}{k\log k}}w(S_k).
$$
where $C$ is an absolute constant. Furthermore, if $K$ is centrally symmetric then
$$
w(K\cap F)\leqslant \sqrt{\frac{n}{k}}w(B_\infty^k).
$$
\end{thm}
We shall also prove the following result on the mean width of projections of convex bodies in L\"owner's position.

\begin{thm}\label{thm:MeanWidthLownerProjections}
Let $K\subseteq\R^n$ be a  convex body in L\"owner's position. Then, for any $k$-dimensional linear subspace $F\in G_{n,k}$,
$$
w(P_F(K))\geq\sqrt{\frac{k}{n}}w(\tilde{S}_k).
$$
Furthermore, if $K$ is centrally symmetric then
$$
w(P_F(K))\geqslant \sqrt{\frac{k}{n}}w(B_1^k).
$$
\end{thm}


For any compact convex set $K\subseteq\R^n$, by Steiner's formula (see \cite[Equation (4.1)]{Sch}) the volume of $K+tB_2^n$ can be expressed as a polynomial in the variable $t$
$$
|K+tB_2^n|=\sum_{i=0}^n{n\choose i}W_i(K)t^i,\quad\forall t\geqslant0,
$$
where the numbers $W_i(K)$ are the so-called querma\ss integrals of $K$. We have that $W_0(K)=|K|$ is the volume of $K$, $nW_1(K)=|\partial K|$ is the surface area of $K$, and $W_{n-1}=|B_2^n|w(K)$, where $w(K)$ is the mean width of $K$. If $K$ is contained in a $k$-dimensional subspace $F\in G_{n,k}$, we can compute its querma\ss integrals in $\R^n$, but also its querma\ss integrals with respect to the subspace $F$, which we identify with $\R^k$. If we denote these querma\ss integrals by $W_i^{(k)}(K)$, for $i=0,\dots,k$, we have that (see e.g.\ \cite[Property~3.1]{SY93})
$$
W_{i}^{(k)}(K)=\frac{{{n}\choose{n-k+i}}}{{{k}\choose{i}}}\frac{|B_2^i|}{|B_{2}^{n-k+i}|}W_{n-k+i}(K),\quad
\forall0\leqslant i\leqslant k,
$$
while $W_i(K)=0$ for all $0\leqslant i<n-k$. In order to avoid the
issue that querma\ss integrals depend on the space where the convex
body is embedded, McMullen \cite{Mc75} defined the  intrinsic volumes
of a compact convex set $K\subseteq\R^n$ as
$$
V_i(K)=\frac{{{n}\choose{i}}}{|B_2^{n-i}|}W_{n-i}(K),\quad
\forall 0\leqslant i\leqslant n.
$$
In \cite{W73} Wills introduced and studied the functional
\begin{equation}\label{e:FWills}
\mathcal{W}(K)=\sum_{i=0}^n V_i(K)
\end{equation}
because of its possible relation with the so-called
lattice-point enumerator $G(K)=\#(K\cap\Z^n)$. In \cite{AHY20}, it was proved that among symmetric convex bodies in John's position, $\mathcal{W}(K)$ is maximized if $K=B_\infty^n$. Here, we prove the following:

\begin{thm}\label{thm:WillsJohnSymmetric}
Let $K\subseteq\R^n$ be a centrally symmetric convex body in John's position. Then, for any $F\in G_{n,k}$ and every $\lambda\geqslant0$,
$$
\mathcal{W}(\lambda(K\cap F))\leqslant \mathcal{W}\left(\lambda\sqrt{\frac{n}{k}}B_\infty^k\right).
$$
\end{thm}

This theorem will give, as direct consequences, the symmetric cases of Theorem \ref{thm:VolumeJohnSections} and Theorem \ref{thm:MeanWidthJohnSetions}. We also prove the following estimate for the Wills functional of projections of convex bodies in L\"owner's position.
\begin{thm}\label{thm:WillsLownerSymmetric}
Let $K\subseteq\R^n$ be a centrally symmetric convex body in L\"owner's position. Then, for any $F\in G_{n,k}$,
$$
\mathcal{W}(P_F(K))\geqslant \frac{1}{k^{k/2}}.
$$
\end{thm}

The main tool used to obtain most of the estimates above is the fact that a decomposition of the identity operator is associated to any convex body in John's position,
and that this decomposition allows the use of Brascamp-Lieb inequality (see \ref{thm:BLIdentity} below). When $K$ is a polytope in minimal surface area, then there is again a decomposition of the identity associated to $K$ (see Section \ref{subsec:MinimalSurfaceArea}. A similar use of Brascamp-Lieb inequality, together with an approximation by polytopes, will lead to similar estimates for sections of convex bodies in minimal surface area position. Namely, we can prove the following

\begin{thm}\label{thm:SurfaceArea}
Let $K\subseteq\R^n$ be a convex body in minimal surface area position. Then, for any $k$-dimensional linear subspace $F\in G_{n,k}$ we have
\begin{enumerate}
\item[(a)]$|\Pi^{*}K\cap F|\leqslant\frac{4^kn^k}{k!}\frac{1}{\partial(K)^k}$,
\item[(b)]$|P_F(\Pi K)|\geqslant \left(\frac{\partial(K)}{n}\right)^k$.
\end{enumerate}
Furthermore, if $K$ is centrally symmetric, then for any $k$-dimensional linear subspace $F\in G_{n,k}$ we have
\begin{enumerate}
\item[(i)]$\mathcal{W}(K\cap F)\leqslant\mathcal{W}\left(\frac{n^2}{k}\frac{|K|}{\partial(K)}B_\infty^k\right)$,
\item[(ii)]$|K\cap F|^{1/k}\leqslant\frac{n^2}{k}\frac{|K|}{\partial(K)}|B_\infty^k|^{1/k}$,
\item[(iii)]$w(K\cap F)\leqslant \frac{n^2}{k}\frac{|K|}{\partial(K)}w(B_\infty^k)$,
\item[(iv)]$|(K\cap F)^\circ|^{1/k}\geqslant \frac{k}{n^2}\frac{\partial(K)}{|K|}|(B_\infty^k)^\circ|^{1/k}$,
\item[(v)]$w((K\cap F)^\circ)\geqslant \frac{k}{n^2}\frac{\partial(K)}{|K|}w((B_\infty^k)^\circ)$.
\end{enumerate}
\end{thm}

\begin{rmk}
Notice that if $k=n$ then (a) recovers the right-hand side of \eqref{eq:PolarProjectionBody}, (b) recovers the left-hand side of \eqref{eq:ProjectionBody}, (ii) recovers the estimate given by Ball's reverse isoperimetric inequality in \cite{B91} and (iii) recovers the estimate given in \cite[Theorem 7.1]{MPS12}.
\end{rmk}

\section{Preliminaries}\label{sec:Preliminaries}

\subsection{John's position}
As mentioned in the introduction, a convex body is said to be in John's position if the maximal volume ellipsoid contained in it is the Euclidean unit ball. A classical theorem of John \cite{J48} (see also \cite{B92}) states that $K$ is in John's position if and only $B_2^n\subseteq K$ and there exist $m=O(n^2)$ contact points $\{u_j\}_{j=1}^m\subseteq\partial K\cap S^{n-1}$ (the intersection of the boundary of $K$ and the Euclidean unit sphere) and $\{c_j\}_{j=1}^m$ with $c_j>0$ for every $1\leqslant j\leqslant m$ such that
\begin{equation}\label{eq:DecompositionIdentity}
I_n=\sum_{j=1}^mc_ju_j\otimes u_j\hspace{1cm}\textrm{and}\hspace{1cm}\sum_{j=1}^mc_ju_j=0,
\end{equation}
where $I_n$ denotes the identity operator in $\R^n$ and $u_j\otimes u_j(x)=\langle x,u_j\rangle u_j$ for every $x\in\R^n$.

Notice that for any such decomposition of the identity, we have that for every $1\leqslant k\leqslant m$
$$
1=|u_k|^2=\sum_{j=1}^mc_j\langle u_k,u_j\rangle^2\geqslant c_k\langle u_k,u_k\rangle^2=c_k,
$$
and so all the numbers $(c_j)_{j=1}^m$ are in the interval $(0,1]$.

\subsection{Brascamp-Lieb inequality}
We will make use of the Brascamp-Lieb inequality and the reverse Brascamp-Lieb inequality in the following form.

\begin{thm}\label{thm:BLIdentity}
Let $m\geq n$, $\{u_j\}_{j=1}^m\subseteq S^{n-1}$, and $\{c_j\}_{j=1}^m\subseteq(0,\infty)$ be such that $\displaystyle{I_n=\sum_{j=1}^mc_j u_j\otimes u_j.}$
Then, for any  integrable functions $\{f_j\}_{j=1}^m:\R\to[0,\infty)$ we have that
$$
\int_{\R^n}\prod_{j=1}^m f_j^{c_j}(\langle x, u_j\rangle)dx\leqslant\prod_{j=1}^m\left(\int_\R f_j(t)dt\right)^{c_j}.
$$
Besides, for any integrable functions $\{h_j\}_{j=1}^m, h:\R\to[0,\infty)$ verifying that\\ $\displaystyle{h\Big(\sum_{j=1}^m\theta_jc_ju_j\Big)\geqslant\prod_{j=1}^mh_j^{c_j}(\theta_j)}$ for every $\{\theta_j\}_{j=1}^m\subseteq\R$, we have that
$$
\int_{\R^n}h(x)dx\geqslant\prod_{j=1}^m\left(\int_\R h_j(t)dt\right)^{c_j}.
$$
\end{thm}

\subsection{The regular simplex}
Let $\Delta_k$ denote the $k$-dimensional regular simplex
$$
\Delta_k=\textrm{conv}\{e_1,\dots,e_{k+1}\}\subseteq H_0,
$$
where $\displaystyle{H_0=\Big\{x\in\R^{k+1}\,:\,\sum_{i=1}^{k+1}x_i=1\Big\}}$ is identified with $\R^k$ and $\left(\frac{1}{k+1},\dots,\frac{1}{k+1}\right)$ is identified with the origin. It is well known that
\begin{itemize}
\item $|\Delta_k|=\frac{\sqrt{k+1}}{k!}$,
\item $r(\Delta_k)=\frac{1}{\sqrt{k(k+1)}}$,
\item $R(\Delta_k)=\sqrt{\frac{k}{k+1}}$,
\item $\Delta_k^\circ=-(k+1)\Delta_k$.
\item $w(\Delta_k)\simeq\sqrt{\frac{\log k}{k}}$,
\end{itemize}
where $a\simeq b$ denotes the fact that there exist two positive absolute constants $c_1,c_2$ such that $c_1a\leqslant b\leqslant c_2a$. Thus, $\frac{1}{r(\Delta_k)}\Delta_k$ is in John's position and $\frac{1}{R(\Delta_k)}\Delta_k$ is in L\"owner's position. Then, if $S_k$ denotes the $k$-dimensional simplex in John's position and $\tilde{S}_k$ denotes the $k$-dimensional simplex in L\"owner's position, we have that
$$
S_k=\sqrt{k(k+1)}\Delta_k\quad\textrm{and}\quad\tilde{S}_k=\sqrt{\frac{k+1}{k}}\Delta_k.
$$
Therefore,
$$
|S_k|^{1/k}=\frac{\sqrt{k(k+1)^{1+\frac{1}{k}}}}{(k!)^{1/k}}\quad\textrm{and}\quad|\tilde{S}_k|^{1/k}=\frac{1}{(k!)^{1/k}}\sqrt{\frac{(k+1)^{1+\frac{1}{k}}}{k}}.
$$
and also
$$
w(S_k)\simeq \sqrt{k\log k}\quad\textrm{and}\quad w(\tilde{S}_k)\simeq \sqrt{\frac{\log k}{k}}.
$$
\subsection{Mean width}

Let $K\subseteq\R^n$ be a convex body. The mean width of $K$ is defined as
$$
w(K)=\int_{S^{n-1}}h_K(\theta)d\sigma(\theta),
$$
where, for every $\theta\in S^{n-1}$,  $h_K(\theta)$ is the support function of $K$ at $\theta$ and $d\sigma$ denotes the uniform probability measure on $S^{n-1}$. If we also assume that $K$ contains the origin in its interior, then $h_K$ is homogeneous of degree $1$ and, integrating in polar coordinates, we have that if $G$ is a standard Gaussian random vector in $\R^n$ then
\begin{eqnarray*}
\E h_K(G)&=&\int_{\R^n}h_K(x) \frac{e^{-\frac{\Vert x\Vert_2^2}{2}}}{(2\pi)^{n/2}}dx=n|B_2^n|\int_0^\infty r^{n}\frac{e^{-\frac{r^2}{2}}}{(2\pi)^{n/2}}\int_{S^{n-1}}h_K(\theta)d\sigma(\theta)\cr
&=&c_n\int_{S^{n-1}}h_K(\theta)d\sigma(\theta)=c_nw(K),
\end{eqnarray*}
where $c_n=\frac{n|B_2^n|\Gamma\left(\frac{n+1}{2}\right)}{\sqrt{2}\pi^{n/2}}=\frac{\sqrt{2}\Gamma\left(\frac{n+1}{2}\right)}{\Gamma\left(\frac{n}{2}\right)}$.
Likewise, since for any convex body containing the origin in its interior the support function of $K^\circ$
is $h_{K^\circ}=\Vert \cdot\Vert_K$, where $\Vert\cdot\Vert_K$ is the Minkowski gauge function of $K$, given by
$$
\Vert x\Vert_K:=\inf\{\lambda>0\,:\,x\in\lambda K\}
$$
for all $x\in\R^n$, we have that if $G$ is a standard Gaussian random vector in $\R^n$
$$
\E \Vert G\Vert_K=c_nw(K^\circ).
$$

\subsection{Log-concave functions}

A function $f\colon \R^n\to[0,\infty)$ is called log-concave if $f(x)=e^{-v(x)}$ where $v:\R^n\to(-\infty,\infty]$ is a convex function. It is well-known that any integrable log-concave function $f\colon \R^n\to[0,\infty)$ is bounded and has moments of all orders. If $K\subseteq\R^n$ is a convex body then its indicator function $\chi_K$ is integrable and log-concave, with integral $|K|$ and if $K$ is a convex body containing the origin, then $e^{-\Vert \cdot\Vert_K}$ is integrable and log-concave, with integral $n!|K|$.

Given a log-concave function $f=e^{-v}$, where $v:\R^n\to(-\infty,\infty]$ is a convex function, its polar function is the function $f^\circ:\R^n\to[0,\infty)$ given by
$$
f^{\circ}(x)=e^{-\mathcal{L}(v)(x)},
$$
where $\mathcal{L}(v)$ denotes the Legendre transform
$$
\mathcal{L}(v)(x)=\sup_{y\in\R^n}(\langle x,y\rangle-v(y)),\quad x\in\R^n.
$$
For more information on log-concave functions we refer the reader to \cite[Chapter 2]{BGVV14}.
\subsection{The Wills functional}\label{subsec:Wills}
Let us recall that for any $n$-dimensional convex body $K$, its Wills functional is defined as
$$
\mathcal{W}(K)=\sum_{i=0}^n V_i(K),
$$
where $V_i(K)$ denotes the $i$-th intrinsic volume of $K$. Many properties of the Wills functional can be found in \cite{W75}, \cite{Ha75}, \cite{Mc91}, or \cite{AHY20}. For our purposes we emphasize the following two:
\begin{enumerate}
\item[(1)](Hadwiger, see \cite[(1.3)]{Ha75}) For any convex body $K\subseteq\R^n$
$$
\mathcal{W}(K)=\int_{\R^n} e^{-\pi d(x,K)^2}dx,
$$
where $d(x, K)$ denotes the Euclidean distance from $x$ to $K$.
\item[(2)](Hadwiger, see \cite[(2.3)]{Ha75}) If $E$ is a linear subspace of $\R^n$, $K_1\subseteq E$ and $K_2\subseteq E^\perp$, then
$$
\mathcal{W}(K_1\times K_2)=\mathcal{W}(K_1)\mathcal{W}(K_2).
$$
\end{enumerate}
In particular, if $K=[-a,a]\subseteq\R$ we have that
$$
\mathcal{W}([-a,a])=2a+2\int_{a}^\infty e^{-\pi (x-a)^2}dx=2a+1
$$
and if $K=aB_\infty^n\subseteq\R^n$ then $\mathcal{W}(aB_\infty^n)=(1+2a)^n$.

Let us point out that for any $\lambda>0$
$$
\mathcal{W}(\lambda K)=\sum_{i=0}^n V_i(\lambda K)=1+\sum_{i=1}^n \lambda^i V_i(K).
$$
Therefore, if two convex bodies $K,L\subseteq\R^n$ verify that $\mathcal{W}(\lambda K)\leqslant\mathcal{W}(\lambda L)$ for every $\lambda\geqslant0$, then one immediately obtains that $V_n(K)\leqslant V_n(L)$ and $V_1(K)\leqslant V_1(L)$ or, equivalently, $|K|\leqslant |L|$ and $w(K)\leqslant w(L)$.

The first property above shows that for any convex body, its Wills functional is the integral of the log-concave function $f_K:\R^n\to[0,\infty)$ given by
$$
f_K(x)=e^{-\pi d(x,K)^2}.
$$

Using a double polarity (both in the convex body and in the family of log-concave functions) we define for any convex body $K\subseteq\R^n$ containing the origin in its interior, the log-concave function $f_{K^\circ}^\circ$. It was proved in \cite[Lemma 3.1]{AHY20} that for every $x\in\R^n$
\begin{equation}\label{eq: PolarFunctionWills}
f_{K^\circ}^\circ(x)=e^{-\frac{\Vert x\Vert_2^2}{4\pi}-\Vert x\Vert_K}.
\end{equation}
The following lemma shows that if, for every $\lambda\geqslant0$, the integral of $f_{(\lambda K)^\circ}^\circ(x)$ is bounded by the integral of $f_{(\lambda L)^\circ}^\circ(x)$, then $|K|\leqslant|L|$.

\begin{lemma}\label{lem:WillsDoublePolarityImpliesVolume}
Let $K,L\subseteq\R^n$ be two convex bodies containing the origin in their interiors. Assume that there exist two numbers $A$  and $\lambda_0>0$ such that,  for any $\lambda\in(0,\lambda_0)$,
$$
\int_{\R^n}f_{(\lambda K)^\circ}^\circ(x)dx\leqslant A\int_{\R^n}f_{(\lambda L)^\circ}^\circ(x)dx.
$$
Then $|K|\leqslant A|L|$.
\end{lemma}

\begin{proof}
Notice that for any convex body $K\subseteq\R^n$ containing the origin in its interior and any $\lambda>0$
\begin{eqnarray*}
\int_{\R^n}f_{(\lambda K)^\circ}^\circ(x)dx&=&\int_{\R^n}e^{-\frac{\Vert x\Vert_2^2}{4\pi}}e^{-\Vert x\Vert_{\lambda K}}dx=\int_{\R^n}e^{-\frac{\Vert x\Vert_2^2}{4\pi}}e^{-\frac{\Vert x\Vert_{K}}{\lambda}}dx\cr
&=&\lambda^n\int_{\R^n}e^{-\frac{\lambda^2\Vert x\Vert_2^2}{4\pi}}e^{-\Vert x\Vert_{K}}dx.
\end{eqnarray*}
Therefore, we have that for every $\lambda\in (0,\lambda_0)$
$$
\int_{\R^n}e^{-\frac{\lambda^2\Vert x\Vert_2^2}{4\pi}}e^{-\Vert x\Vert_{K}}dx\leqslant A\int_{\R^n}e^{-\frac{\lambda^2\Vert x\Vert_2^2}{4\pi}}e^{-\Vert x\Vert_{L}}dx
$$
and, taking the limit as $\lambda$ tends to $0$ we obtain that
$$
n!|K|=\int_{\R^n}e^{-\Vert x\Vert_{K}}dx\leqslant A\int_{\R^n}e^{-\Vert x\Vert_{L}}dx=n!A|L|.
$$
\end{proof}
\subsection{Convex bodies in minimal surface area position}\label{subsec:MinimalSurfaceArea}

A convex body $K\subseteq\R^n$ is said to be in minimal surface area position if it has minimal surface area among all of its volume preserving affine images. That is, if
$$
\partial(K)=\min\left\{\partial (TK)\,:\,T\in SL(n)\right\},
$$
where $SL(n)$ denotes the set of non-degenerate linear maps $T\in GL(n)$ with $|\textrm{det}T|=1$. The surface area measure of a convex body $K$ is the measure on the sphere defined by
$$
\sigma_K(A):=\nu\left(\{x\in\partial K\,:\,\nu_K(x)\in A\}\right)\quad\forall A\textrm{ Borel set in }S^{n-1},
$$
where $\nu$ denotes the Hausdorff measure on $\partial K$ and $\nu_K(x)$ is the outer normal vector to $K$ at $x$, which is defined except for a set of measure $\nu$ equal to $0$.

The projection body $\Pi K$ and its polar, the polar projection body $\Pi^*K$, of a convex body $K$ are the centrally symmetric convex bodies defined by
$$
h_{\Pi K}(x)=\Vert x\Vert_{\Pi^*(K)}=|x||P_{x^\perp}K|=\frac{1}{2}\int_{S^{n-1}}|\langle x,\theta\rangle|d\sigma_K(\theta),
$$
where for any $x\neq 0$, $|P_{x^\perp}K|$ denotes the $(n-1)$-dimensional volume of the projection of $K$ onto the hyperplane orthogonal to $x$ and the last equality is the well-known Cauchy's formula (see, for instance, \cite[Equation (5.80)]{Sch}).

It was proved by Petty \cite{P61} (see also \cite{GP99}) that $K$ is in minimal surface area position if and only if $\sigma_K$ is isotropic, i.e., if
$$
I_n=\frac{n}{\partial K}\int_{S^{n-1}}u\otimes ud\sigma_K(u),
$$
and it was observed in \cite{GMR00} that the latter happens if and only if $\Pi K$ is in minimal mean width position, i.e.,
$$
w(\Pi K)=\min\{w(T\Pi K)\,:\,T\in SL(n)\}.
$$
Notice that if $K$ is a polytope with facets $\{F_j\}_{j=1}^m$ with outer normal vectors $\{u_j\}_{j=1}^m$ then the surface area measure of $K$ is
$$
\sigma_K=\sum_{j=1}^m|F_j|\delta_{u_j},
$$
where $\delta_j$ denotes the Dirac delta measure on $u_j$ and $K$ is in minimal surface area position if and only if
$$
I_n=\sum_{j=1}^m\frac{n|F_j|}{\partial(K)}u_j\otimes u_j.
$$
 In particular, if $K$ is a polytope with facets $\{F_j\}_{j=1}^m$ and outer normal vectors $\{u_j\}_{j=1}^m$ then for every $x\in\R^n$
$$
h_{\Pi K}(x)= \Vert x\Vert_{\Pi^*(K)}=\frac{1}{2}\sum_{j=1}^m|F_j||\langle x,u_j\rangle|.
$$
It was proved in \cite{GP99} that, as a consequence of a lemma obtained from the Brascamp-Lieb inequality (see \cite{B91}), if $K$ is a convex body in minimal surface area position then
\begin{equation}\label{eq:PolarProjectionBody}
|B_2^n|\left(\frac{n|B_2^n|}{|B_2^{n-1}|}\right)^n\frac{1}{\partial(K)^n}\leqslant|\Pi^*K|\leqslant\frac{4^nn^n}{n!}\frac{1}{\partial(K)^n}
\end{equation}
and, as a consequence of the Blaschke-Santal\'o inequality and its exact reverse for zonoids (see \cite{GMR88} and \cite{R86}), or as a direct consequence of the reverse form of Brascamp-lieb inequality (see \cite{GMR00}),
\begin{equation}\label{eq:ProjectionBody}
\left(\frac{\partial(K)}{n}\right)^n\leqslant|\Pi K|\leqslant|B_2^n|\left(\frac{|B_2^{n-1}|}{n|B_2^n|}\right)^n\partial(K)^n.
\end{equation}

\section{General setting}\label{sec:GeneralSetting}

In this section we introduce the notation for a setting that will be used in several of our proofs. Let $K\subseteq\R^n$ be a (not necessarily symmetric) convex body. Then there exist $\{u_j\}_{j=1}^m\subseteq \partial K\cap S^{n-1}$ and $\{c_j\}_{j=1}^m\subseteq(0,\infty)$ such that
$$
I_n=\sum_{j=1}^mc_j u_j\otimes u_j,\quad\quad\sum_{j=1}^m c_ju_j=0\quad\textrm{and}\quad\sum_{j=1}^m c_j=n,
$$
where the third equality is obtained from the first one by taking traces. We will denote by $C\subseteq\R^n$ the convex body
$$
C=\{x\in\R^n\,:\, \langle x, u_j\rangle\leqslant 1,\,\forall 1\leqslant j\leqslant m\}.
$$
It is easily verified that $K\subseteq C$. We will denote, for every $1\leqslant j\leqslant m$,
\begin{itemize}
\item $v_j=\sqrt{\frac{n}{n+1}}(-u_j,\frac{1}{\sqrt{n}})\in S^{n}$, and
\item $\delta_j=\frac{n+1}{n}c_j\in(0,1]$.
\end{itemize}
These vectors satisfy
$$
I_{n+1}=\sum_{j=1}^m\delta_j v_j\otimes v_j,\quad\quad\sum_{j=1}^m \delta_jv_j=\left(0,\sqrt{n+1}\right)\quad\textrm{and}\quad\sum_{j=1}^m \delta_j=n+1.
$$
We will denote by $L$ the cone
$$
L=\{y=(x,r)\in\R^{n+1}\,:\,\langle y, v_j\rangle\geq0,\,\forall\, 1\leqslant j\leqslant m\}.
$$
The next lemma relates $L$ and $C$.
\begin{lemma}\label{lem:ConeL}
Let $K\subseteq\R^n$ be a convex body in John's position and let $L$ be defined as above. Then
$$
L=\left\{(x,r)\in \R^{n+1}\,:\,r\geqslant 0, x\in \frac{r}{\sqrt{n}}C\right\}.
$$
\end{lemma}
\begin{proof}
Let $y=(x,r)\in L$. By the definition of $v_j$ we have that for each $1\leqslant j\leqslant m$
$$
\langle y,v_j\rangle =-\sqrt{\frac{n}{n+1}}\langle x,u_j\rangle+\frac{r}{\sqrt{n+1}}.
$$
Assume that $r<0$. Then, since  $\langle y,v_j\rangle \geqslant 0$ for every $1\leqslant j\leqslant m$ we have that
$$
-\sqrt{\frac{n}{n+1}}\langle x,u_j\rangle+\frac{r}{\sqrt{n+1}}\geqslant 0\quad\forall 1\leqslant j\leqslant m
$$
and then $\langle x,u_j\rangle<0$ for every $1\leqslant j\leqslant m$. As a consequence, since $\{c_j\}_{j=1}^m\subseteq(0,\infty)$,
$$
\sum_{j\in J} c_j\langle x,u_j\rangle <0,
$$
which contradicts the fact that $\displaystyle{\sum_{j=1}^mc_ju_j=0}$. Therefore, if $y=(x,r)\in L$ then $r\geqslant 0$ and for any $r\geqslant 0$ we have $\langle y,v_j\rangle \geqslant 0$ for every $1\leqslant j\leqslant m$ if and only if
$\langle x,u_j\rangle\leqslant\frac{r}{\sqrt{n}}$ for every $1\leqslant j\leqslant m$.
The latter condition is true if and only if $x\in \frac{r}{\sqrt{n}}C$.

Conversely, if $y=(x,r)$ verifies that $r\geqslant0$ and $x\in\frac{r}{\sqrt{n}}C$, which happens if and only if $\langle x,u_j\rangle\leqslant\frac{r}{\sqrt{n}}$ for every $1\leqslant j\leqslant m$, then for every $1\leqslant j\leqslant m$
$$
\langle y, v_j\rangle=-\sqrt{\frac{n}{n+1}}\langle x,u_j\rangle+\frac{r}{\sqrt{n+1}}\geqslant 0
$$
and then $y\in L$.
\end{proof}
Given any $k$-dimensional affine subspace $F$ in $\R^n$, we will consider the linear $(k+1)$-dimensional subspace in $\R^{n+1}$
$$
H=\textrm{span}\{\left(x,\sqrt{n}\right)\,:\,x\in F\}.
$$
Notice that if $F\in G_{n,k}$ is a linear subspace then $H=F\times\R$. We set $J=\{1\leqslant j\leqslant m\,:\,P_H v_j\neq0\}$ and, for every $j\in J$, we define
\begin{itemize}
\item $w_j=\frac{P_H v_j}{\Vert P_Hv_j\Vert_2}$,
\item $\kappa_j=\delta_j\Vert P_Hv_j\Vert_2^2=\frac{n+1}{n}c_j\Vert P_H v_j\Vert_2^2$.
\end{itemize}
Then, we have that
$$
I_H=\sum_{j\in J}\kappa_jw_j\otimes w_j\quad\textrm{and}\quad\sum_{j\in J}\kappa_j=k+1,
$$
where $I_H$ denotes the identity in the linear subspace $H$. Furthermore, if for every $j\in J$ we denote by $s_j=\frac{1}{\Vert P_H v_j\Vert_2}$ we have that for every $y=(x,r)\in H\subseteq\R^{n+1}$
\begin{align}\label{eq:SumInH}
\sum_{j\in J}\kappa_j s_j\langle y, w_j\rangle&=\sum_{j\in J}\delta_j\langle y,P_H v_j\rangle=\sum_{j=1}^m\delta_j\langle y,P_H v_j\rangle=\sum_{j=1}^m\delta_j\langle y,v_j\rangle\\
\nonumber &=r\sqrt{n+1}.
\end{align}
The following lemma shows that if $F\in G_{n,k}$ is a linear subspace, and if $H=F\times\R$, then we have that $J=\{1,\dots, m\}$ and there is a lower bound for the Euclidean norm of $P_Hv_j$ for every $1\leqslant j\leqslant m$.

\begin{lemma}\label{lem:LowerBoundLinear}
Let $\{u_j\}_{j=1}^m\subseteq S^{n-1}$, $\{c_j\}_{j=1}^m$ be such that \eqref{eq:DecompositionIdentity} holds, $F\in G_{n,k}$, $H=F\times\R\in G_{n+1,k+1}$, and $\{v_j\}_{j=1}^m\in S^{n}$ be defined as above. Then, for every $1\leqslant j\leqslant m$ we have
$$
\frac{1}{n+1}\leqslant \Vert P_H v_j\Vert_2^2\leqslant 1.
$$
\end{lemma}
\begin{proof}
Let $c=\left(0,\frac{1}{\sqrt{n+1}}\right)\in H$ and notice that for every $1\leqslant j\leqslant m$
$$
\langle P_H(v_j-c),c\rangle=\langle v_j-c,c\rangle=\frac{1}{n+1}-\frac{1}{n+1}=0
$$
and then, since $c\in H$,
$$
\Vert P_H v_j\Vert_2^2=\Vert c+(P_H(v_j-c))\Vert_2^2=\Vert c\Vert_2^2+\Vert P_H(v_j-c)\Vert_2^2\geq\Vert c\Vert_2^2=\frac{1}{n+1}.
$$
Thus, for every $1\leqslant j\leqslant m$ we have that
$$
\frac{1}{n+1}\leqslant \Vert P_H v_j\Vert_2^2\leqslant 1.
$$
\end{proof}

If $K\subseteq\R^n$ is a centrally symmetric convex body in John's position then we will also denote by $C_0$ the symmetric convex body
$$
C_0=\{x\in\R^n\,:\, |\langle x, u_j\rangle|\leqslant 1,\,\forall \,1\leqslant j\leqslant m\},
$$
which verifies that $K\subseteq C_0$. If $F\in G_{n,k}$ is a linear subspace, we set $J_0=\{1\leqslant j\leqslant m\,:\,P_F u_j\neq0\}$ and for every $j\in J_0$, we define
\begin{itemize}
\item $v_j^0=\frac{P_F u_j}{\Vert P_F u_j\Vert_2}\in S_F$,
\item $\delta_j^0= c_j\Vert P_F u_j\Vert_2^2$.
\end{itemize}
Then, we have that
$$
I_F=\sum_{j\in J_0}^m\delta_j^0 v_j^0\otimes v_j^0\hspace{1cm}\textrm{and}\hspace{1cm}\sum_{j\in J_0}\delta_j^0=k,
$$
where $I_F$ denotes the identity operator in $F$, and also
\begin{eqnarray*}
K\cap F\subseteq C_0\cap F&=&\{x\in F\,:\, |\langle x, u_j\rangle|\leqslant 1,\,\forall 1\leqslant j\leqslant m\}\cr
&=&\{x\in F\,:\, |\langle x, P_Fu_j\rangle|\leqslant 1,\,\forall 1\leqslant j\leqslant m\}\cr
&=&\{x\in F\,:\, |\langle x, P_Fu_j\rangle|\leqslant 1,\,\forall j\in J_0\}\cr
&=&\left\{x\in F\,:\, |\langle x, v_j^0\rangle|\leqslant t_j,\,\forall j\in J_0\right\},\cr
\end{eqnarray*}
where for every $j\in J_0$ we have set $t_j=\frac{1}{\Vert P_F u_j\Vert_2}=\left(\frac{c_j}{\delta_j^0}\right)^{1/2}$. Furthermore,
\begin{eqnarray*}
(K\cap F)^\circ&\supseteq&(C_0\cap F)^\circ=P_F(C_0^\circ)=P_F\left(\textrm{conv}\{\pm u_j\,:\,1\leqslant j\leqslant m\}\right)\cr
&=&\textrm{conv}\{\pm P_F u_j \,:\,1\leqslant j\leqslant m\}\cr
&=&\textrm{conv}\{\pm P_F u_j \,:\,j\in J_0\}.
\end{eqnarray*}

Let $K$ be a (not necessarily centrally symmetric) polytope in minimal surface area position with facets $\{F_j\}_{j=1}^m$ and outer normal vectors $\{u_j\}_{j=1}^m$, and let $F\in G_{n,k}$ be a $k$-dimensional linear subspace. Then,
$$
K=\{x\in\R^n\,:\,\langle x, u_j\rangle\leqslant h_K(u_j),\,\forall\,1\leqslant j\leqslant m\}
$$
and
$$
I_n=\sum_{j=1}^m\frac{n|F_j|}{\partial(K)}u_j\otimes u_j=\sum_{j=1}^mc_ju_j\otimes u_j,
$$
where $c_j=\frac{n|F_j|}{\partial(K)}$ for every $1\leqslant j\leqslant m$. Besides (see, for instance, \cite[Theorem 18.2]{Gr07})
$$
\sum_{j=1}^mc_ju_j=\frac{n}{\partial(K)}\sum_{j=1}^m|F_j|u_j=0.
$$
and
\begin{equation}\label{eq:VolumeMinimalSurfaceArea}
\sum_{j=1}c_jh_K(u_j)=\sum_{j=1}^m\frac{n|F_j|}{\partial(K)}h_K(u_j)=\frac{n^2|K|}{\partial(K)}.
\end{equation}
Note also that if $K$ is a centrally symmetric polytope in minimal surface area position, with facets $\{F_j\}_{j=1}^m$ and outer normal vectors $\{u_j\}_{j=1}^m$, and if $F\in G_{n,k}$ is a $k$-dimensional linear subspace, then
$$
K=\{x\in\R^n\,:\,|\langle x, u_j\rangle|\leqslant h_K(u_j),\,\forall\,1\leqslant j\leqslant m\}.
$$
As in the case where the decomposition of the identity comes from a centrally symmetric convex body in John's position, we set $J_0=\{1\leqslant j\leqslant m\,:\,P_Fu_j\neq0\}$ and, for every $j\in J_0$, we define
\begin{itemize}
\item  $v_j^0=\frac{P_Fu_j}{\Vert P_F u_j\Vert_2}$,
\item $\delta_j^0=c_j\Vert P_F u_j\Vert_2^2=\frac{n|F_j|\|P_Fu_j\|_2^2}{\partial(K)}$.
\end{itemize}
We have that
$$I_F=\sum_{j=1}^mc_jP_Fu_j\otimes P_Fu_j=\sum_{j\in J_0}\delta_j^0v_j^0\otimes v_j^0.$$
Besides, if we denote $t_j=\frac{1}{\Vert P_Fu_j\Vert_2}=\left(\frac{c_j}{\delta_j^0}\right)^{1/2}$ for every $j\in J_0$, then
$$
K\cap F=\{x\in F\,:\,|\langle x,v_j^0\rangle|\leqslant t_jh_K(u_j),\forall\, j\in J_0\}
$$
and
\begin{eqnarray*}
(K\cap F)^\circ&=&\textrm{conv}\left\{\pm\frac{P_F(u_j)}{h_K(u_j)}\,:\,1\leqslant j\leqslant m\right\}\cr
&=&\textrm{conv}\left\{\pm\frac{P_F(u_j)}{h_K(u_j)}\,:\,j\in J_0\right\}.\cr
\end{eqnarray*}
\section{Volume of sections of convex bodies in John's position}\label{sec:VolumeSections}

In this section we will give the proof of Theorem \ref{thm:VolumeJohnSections}.

\begin{proof}[Proof of Theorem \ref{thm:VolumeJohnSections}.] Let us first start with the symmetric case. Assume that $K$ is a centrally symmetric convex body in John's position and $F\in G_{n,k}$ is a linear $k$-dimensional subspace. Following the notation in Section \ref{sec:GeneralSetting}, we have that

$$
K\cap F\subseteq C_0\cap F=\left\{x\in F\,:\, |\langle x, v_j^0\rangle|\leqslant t_j,\,\forall j\in J_0\right\},
$$
where for every $j\in J_0$ we have denoted $t_j=\frac{1}{\Vert P_F u_j\Vert_2}=\left(\frac{c_j}{\delta_j^0}\right)^{1/2}$. Therefore, by
the Brascamp-Lieb inequality,
\begin{eqnarray*}
|K\cap F|&\leqslant&|C_0\cap F|=\int_F\prod_{j\in J_0}\chi_{[-t_j,t_j]}(\langle x,v_j^0\rangle)dx=\int_F\prod_{j\in J_0}\chi_{[-t_j,t_j]}^{\delta_j^0}(\langle x,v_j^0\rangle)dx\cr
&\leqslant&\prod_{j\in J_0}\left(\int_{\R}\chi_{[-t_j,t_j]}(t)\right)^{\delta_j^0}=\prod_{j\in J_0}(2t_j)^{\delta_j^0}=2^k\prod_{j\in J_0}\left(\frac{c_j}{\delta_j^0}\right)^{\frac{\delta_j^0}{2}}.
\end{eqnarray*}
By the arithmetic-geometric mean inequality we get
$$
\prod_{j\in J_0}\left(\frac{c_j}{\delta_j^0}\right)^{\frac{\delta_j^0}{k}}\leqslant\sum_{j\in J_0}\frac{\delta_j^0}{k}\frac{c_j}{\delta_j^0}=\frac{1}{k}\sum_{j\in J_0}c_j\leqslant\frac{1}{k}\sum_{j=1}^mc_j=\frac{n}{k},
$$
and hence,
$$
|K\cap F|^{1/k}\leqslant 2\sqrt{\frac{n}{k}}=\sqrt{\frac{n}{k}}|B_\infty^k|^{1/k}.
$$
Assume now that $K\subseteq\R^n$ is a (not necessarily symmetric) convex body in John's position and $F\in G_{n,k}$ a linear $k$-dimensional subspace. Following the notation introduced in Section \ref{sec:GeneralSetting} we have that if we take $H=F\times\R\in G_{n+1,k+1}$ then
$$
L\cap H=\left\{(x,r)\in F\times\R\,:\,r\geqslant 0, x\in \frac{r}{\sqrt{n}}(C\cap F)\right\}.
$$
Denoting $s_j=\frac{1}{\Vert P_Hv_j\Vert_2}$  we have, by the Brascamp-Lieb inequality, that
\begin{eqnarray*}
\int_{L\cap H} e^{-\sum_{j=1}^m\kappa_js_j\langle y,w_j\rangle}dy&=&\int_H\prod_{j=1}^m\left(\chi_{[0,\infty)}(\langle y,v_j\rangle)e^{-s_j\langle y,w_j\rangle}\right)^{\kappa_j}dy\cr
&=&\int_H\prod_{j=1}^m\left(\chi_{[0,\infty)}(\langle y,P_Hv_j\rangle)e^{-s_j\langle y,w_j\rangle}\right)^{\kappa_j}dy\cr
&=&\int_H\prod_{j=1}^m\left(\chi_{[0,\infty)}(\langle y,w_j\rangle)e^{-s_j\langle y,w_j\rangle}\right)^{\kappa_j}dy\cr
&\leqslant&\prod_{j=1}^m\left(\int_0^\infty e^{-s_jt}dt\right)^{\kappa_j}\cr
&=&\prod_{j=1}^m\Vert P_Hv_j\Vert_2^{\delta_j\Vert P_H v_j\Vert_2^2}.
\end{eqnarray*}
On the other hand, taking into account \eqref{eq:SumInH} we see that
\begin{eqnarray*}
\int_{L\cap H} e^{-\sum_{j=1}^m\kappa_js_j\langle y,w_j\rangle}dy&=&\int_0^\infty\int_{\frac{r}{\sqrt{n}}(C\cap F)}e^{-r\sqrt{n+1}}dxdr\cr
&=&\int_0^\infty \frac{r^k}{n^{\frac{k}{2}}}|C\cap F|e^{-\sqrt{n+1}r}dr\cr
&=&\frac{k!}{n^\frac{k}{2}(n+1)^\frac{k+1}{2}}|C\cap F|\cr
&=&\frac{k^\frac{k}{2}(k+1)^\frac{k+1}{2}}{n^\frac{k}{2}(n+1)^\frac{k+1}{2}}\frac{|C\cap F|}{|S_k|}.
\end{eqnarray*}
Let us maximize $\prod_{j=1}^m\Vert P_Hv_j\Vert_2^{\delta_j\Vert P_H v_j\Vert_2^2}$ under the constraints
\begin{itemize}
\item $ \frac{1}{n+1}\leqslant \Vert P_H v_j\Vert_2^2\leqslant 1$ $\forall 1\leqslant j\leqslant m$,
\item $\sum_{j=1}^m \delta_j\Vert P_H v_j\Vert_2^2=k+1$,
\item $\sum_{j=1}^m\delta_j=n+1$,
\item $0\leqslant\delta_j\leqslant1$.
\end{itemize}
Equivalently, let us maximize $F(x,\delta)=\frac{1}{2}\sum_{j=1}^m\delta_jx_j\log x_j$ under the constraints
\begin{itemize}
\item $ \frac{1}{n+1}\leqslant x_j\leqslant 1$ $\forall 1\leqslant j\leqslant m$,
\item $\sum_{j=1}^m \delta_jx_j=k+1$,
\item $\sum_{j=1}^m\delta_j=n+1$,
\item $0\leqslant\delta_j\leqslant1$.
\end{itemize}
First notice that the function $F(x,\delta)$ is continuous on the compact domain in $M\subseteq\R^{2m}$ given by the constraints and, therefore, it  attains its maximum. For every $x=(x_1,\dots, x_m)$ with $\frac{1}{n+1}\leqslant x_j\leqslant 1$ with $1\leqslant j\leqslant m$, for all $1\leqslant j\leqslant m$, let $F_x(\delta)$ be the function
$$
F_x(\delta)=\frac{1}{2}\sum_{j=1}^m\delta_jx_j\log x_j.
$$
Notice that $F_x$ is a convex function and then, since the set
$$
A=\Big\{\delta\in\R^m\,:\,\sum_{j=1}^m \delta_jx_j=k+1,\,\sum_{j=1}^m\delta_j=n+1,\,0\leqslant\delta_j\leqslant1\,\forall 1\leqslant j\leqslant m\Big\}
$$
is a compact convex set, $F_x$ attains its maximum on some extreme point of $A$. These are the points of intersection of the $2$-dimensional faces of the cube $\{\delta\in\R^m\,:\,0\leqslant \delta_j\leqslant 1\,\;\forall 1\leqslant j\leqslant m\}$ with the $(m-2)$-dimensional affine subspace
$$
\Big\{\delta\in\R^m\,:\,\sum_{j=1}^m \delta_jx_j=k+1,\,\sum_{j=1}^m\delta_j=n+1\Big\}.
$$
Therefore, a maximizer of the function $F_x$ has to be a point of the form
$$\delta_\lambda=(\underbrace{1,1,\ldots,1}_{n },\lambda, 1-\lambda,\underbrace{0,\dots,0}_{m-n-2})$$
for some $\frac{1}{2}\leqslant\lambda\leqslant1$ (or a permutation of it) such that $\sum_{j=1}^m \delta_jx_j=k+1$ is satisfied. Let us find, for every  $\delta_\lambda$ with $\frac{1}{2}\leqslant\lambda\leqslant1$, the maximizer of the function
$$
F_{\delta_\lambda}(x)=\frac{1}{2}\sum_{j=1}^m\delta_jx_j\log x_j
$$
on the compact convex set
$$
B_\lambda=\Big\{x\in\R^m\,:\,\sum_{j=1}^m \delta_{\lambda,j}x_j=k+1,\,\frac{1}{n+1}\leqslant x_j\leqslant 1\;\;\forall 1\leqslant j\leqslant m\Big\}.
$$
We can assume, without loss of generality, that $\delta_\lambda=\delta_\lambda^*$, where $\delta_\lambda^*$ is the decreasing rearrangement of $\delta_\lambda$. Let $$C=\frac{k+1}{n+1}$$ and
$$
\tilde{x}=\left(\underbrace{1,1,\ldots,1}_{k },C,\underbrace{\frac{1}{n+1},\ldots,\frac{1}{n+1}}_{m-k-1}\right).
$$
We check that $\frac{1}{n+1}\leqslant C\leqslant1$. Moreover,
$$
\sum_{j=1}^m\delta_{\lambda, j}\tilde{x}_j=k+C+\frac{n-k}{n+1}=k+1
$$
and for every $x=(x_1,\dots,x_m)\in B_\lambda$, since the first $k+1$ coordinates of $\tilde{x}$ are as large as they can,
we have $\tilde{x}\succ (x_1,\ldots, x_m)$, i.e.,
\begin{itemize}
\item $\displaystyle{\sum_{j=1}^m\delta_{\lambda,j}\tilde{x}_j=\sum_{j=1}^m\delta_{\lambda, j}x_j=k+1}$,
\item $\displaystyle{\sum_{j=1}^l\delta_{\lambda,j}\tilde{x}_j\geqslant\sum_{j=1}^l\delta_{\lambda, j}x_j\,\;\forall 1\leqslant l\leqslant m}$.
\end{itemize}
Therefore, by the weighted Karamata's inequality we have that, for every $x\in B_\lambda$,
$$F_{\delta_\lambda}(x)\leqslant F_{\delta_\lambda}(\tilde{x})\leqslant\max_{(\delta,x)\in M}F(\delta,x)$$
and then, since
$$
\max_{(\delta,x)\in M}F(\delta,x)\leqslant\max_{\lambda\in[\frac{1}{2},1],x\in B_\lambda}F_{\delta_\lambda}(x)
$$
we see that
\begin{eqnarray*}
\max_{(\delta,x)\in M}F(\delta,x)&=&\max_{\lambda\in[\frac{1}{2},1]} F_{\delta_\lambda}(\tilde{x})=\max_{\lambda\in[\frac{1}{2},1]}\left\{\frac{1}{2}C\log C+\frac{n-k}{2(n+1)}\log\left(\frac{1}{n+1}\right)\right\}\cr
&=&\frac{1}{2}C\log C-\frac{n-k}{n+1}\log(n+1).
\end{eqnarray*}
Thus,
$$
\prod_{j=1}^m\Vert P_Hv_j\Vert_2^{\delta_j\Vert P_H v_j\Vert_2^2}\leqslant e^{\frac{1}{2}C\log C-\frac{n-k}{2(n+1)}\log(n+1)}=\frac{\left(k+1\right)^\frac{k+1}{2(n+1)}}{(n+1)^\frac{n+1}{2(n+1)}}.
$$
Therefore, since $|K\cap F|\leq|C\cap F|$ we have that
$$
|K\cap F|^{1/k}\leqslant\frac{1}{(k+1)^{\frac{n-k}{2k(n+1)}}}\sqrt{\frac{n}{k}\frac{(n+1)}{(k+1)}}|S_k|^{1/k}.
$$
Finally, assume now that $K\subseteq\R^n$ is a (not necessarily symmetric) convex body in John's position and $F$ is a $k$-dimensional affine subspace at distance $d$ from $0$. Following the notation introduced in Section \ref{sec:GeneralSetting}, given the $k$-dimensional affine subspace $F$, if we take the linear subspace $H=\textrm{span}\{(x,\sqrt{n})\,:\,x\in F_1\}\in G_{n+1,k+1}$ we have that
$$
L\cap H=\left\{(x,r)\in \R^{n+1}\,:\,r\geqslant 0, x\in \frac{r}{\sqrt{n}}(C\cap F)\right\}.
$$
Setting $J=\{1\leqslant j\leqslant m\,:\,P_H v_j\neq0\}$ and $s_j=\frac{1}{\Vert P_Hv_j\Vert_2}$ we have, by the Brascamp-Lieb inequality, that
\begin{eqnarray*}
\int_{L\cap H} e^{-\sum_{j\in J}\kappa_js_j\langle y,w_j\rangle}dy&=&\int_H\prod_{j\in J}\left(\chi_{[0,\infty)}(\langle y,v_j\rangle)e^{-s_j\langle y,w_j\rangle}\right)^{\kappa_j}dy\cr
&=&\int_H\prod_{j\in J}\left(\chi_{[0,\infty)}(\langle y,P_Hv_j\rangle)e^{-s_j\langle y,w_j\rangle}\right)^{\kappa_j}dy\cr
&=&\int_H\prod_{j\in J}\left(\chi_{[0,\infty)}(\langle y,w_j\rangle)e^{-s_j\langle y,w_j\rangle}\right)^{\kappa_j}dy\cr
&\leqslant&\prod_{j\in J}\left(\int_0^\infty e^{-s_jt}dt\right)^{\kappa_j}=\prod_{j\in J}\Vert P_Hv_j\Vert_2^{\delta_j\Vert P_H v_j\Vert_2^2}\cr
&\leqslant&1.
\end{eqnarray*}
Taking into account \eqref{eq:SumInH} we have that
\begin{eqnarray*}
\int_{L\cap H} e^{-\sum_{j\in J}\kappa_js_j\langle y,w_j\rangle}dy&=&\int_0^\infty\int_{\frac{r}{\sqrt{n}}(C\cap F)}e^{-\sqrt{n+1}r}dx\sqrt{\frac{n+d^2}{n}}dr\cr
&=&\int_0^\infty \frac{r^k(n+d^2)^\frac{1}{2}}{n^{\frac{k+1}{2}}}|C\cap F|e^{-\sqrt{n+1}r}dr\cr
&=&\frac{(n+d^2)^{\frac{1}{2}}k!}{n^\frac{k+1}{2}(n+1)^\frac{k+1}{2}}|C\cap F|.\cr
\end{eqnarray*}
Since $|K\cap F|\leq|C\cap F|$, we get
$$
|K\cap F|\leqslant\frac{n^\frac{k}{2}(n+1)^\frac{k+1}{2}}{k!}\sqrt{\frac{n}{n+d^2}}
$$
or, equivalently,
$$
|K\cap F|^{1/k}\leqslant\sqrt{\frac{n(n+1)^{1+\frac{1}{k}}}{k(k+1)^{1+\frac{1}{k}}}}\left(\frac{n}{n+d^2}\right)^{\frac{1}{2k}}|S_k|^{1/k}.
$$
\end{proof}

\section{Volume of projections of convex bodies in L\"owner's position}\label{sec:VolumeProjections}
In this section we will give the proof of Theorem \ref{thm:VolumeLownerProjections}.

\begin{proof}[Proof of Theorem \ref{thm:VolumeLownerProjections}.] Let us start with the symmetric case. Assume that $K$ is a centrally symmetric convex body in John's position. Following the notation in Section \ref{sec:GeneralSetting} we have that $K\subseteq C_0$, therefore $K\cap F\subseteq C_0\cap F$. This gives that
$$
(K\cap F)^\circ\supseteq (C_0\cap F)^\circ=\textrm{conv}\{\pm P_F u_j\,:\, j\in J_0\}.
$$
Therefore, for every $x\in F$ we have that
\begin{eqnarray*}
h_{K\cap F}(x)&\leqslant& h_{C_0\cap F}(x)=\Vert x\Vert_{(C_0\cap F)^\circ}=\inf\left\{\sum_{j\in J_0}|\alpha_j|\,:\,x=\sum_{j\in J_0}\alpha_jP_Fu_j\right\}\cr
&=&\inf\left\{\sum_{j\in J_0}|\alpha_j|\,:\,x=\sum_{j\in J_0}\alpha_j\Vert P_Fu_j\Vert_2v_j^0\right\}\cr
&=&\inf\left\{\sum_{j\in J_0}\frac{|\beta_j|}{\Vert P_F u_j\Vert_2}\,:\,x=\sum_{j\in J_0}\beta_jv_j^0\right\}\cr
&=&\inf\left\{\sum_{j\in J_0}\delta_j^0|\theta_j|t_j\,:\,x=\sum_{j\in J_0}\delta_j^0\theta_jv_j^0\right\},
\end{eqnarray*}
where we have set, for every $j\in J_0$, $t_j=\frac{1}{\Vert P_H u_j\Vert_2}=\left(\frac{c_j}{\delta_j^0}\right)^\frac{1}{2}$. For every $j\in J_0$, we set
$$
f_j(t):=e^{-|t|t_j},\ \ t\in\R.
$$
Then, if $\displaystyle{x=\sum_{j\in J_0}\delta_j^0\theta_jv_j^0}$ for some $\{\theta_j\}_{j\in J_0}\subseteq \R$, we have
$$
\prod_{j\in J_0} f_j^{\delta_j^0}(\theta_j)=e^{-\sum_{j\in J_0}\delta_j|\theta_j|t_j}\leqslant e^{-h_{K\cap F}(x)}.
$$
Therefore, by the reverse Brascamp-Lieb inequality we obtain
\begin{eqnarray*}
k!|(K\cap F)^\circ|&=&\int_{F}e^{-h_{K\cap F}(x)}dx\geqslant\prod_{j\in J}\left(\int_{\R}e^{-|t|t_j}dt\right)^{\delta_j}\cr
&=&\frac{2^k}{\prod_{j\in J_0}t_j^{\delta_j}}=\frac{2^k}{\prod_{j\in J_0}\left(\frac{c_j}{\delta_j}\right)^\frac{\delta_j}{2}}.
\end{eqnarray*}
As we have seen in the proof of Theorem \ref{thm:VolumeJohnSections}
$$
\prod_{j\in J_0}\left(\frac{c_j}{\delta_j^0}\right)^{\frac{\delta_j}{k}}\leqslant\sum_{j\in J_0}\frac{\delta_j^0}{k}\frac{c_j}{\delta_j^0}=\frac{1}{k}\sum_{j\in J_0}c_j\leqslant\frac{1}{k}\sum_{j=1}^mc_j=\frac{n}{k}.
$$
Taking into account that $|(B_\infty^k)^\circ|=|B_1^k|=\frac{2^k}{k!}$, we obtain
$$
|(K\cap F)^\circ|^{1/k}\geqslant\sqrt{\frac{k}{n}}|(B_\infty^k)^\circ|^{1/k}.
$$

Assume now that $K\subseteq\R^n$ is a (not necessarily symmetric) convex body in John's position and $F\in G_{n,k}$ is a $k$-dimensional linear subspace. Following the notation in Section \ref{sec:GeneralSetting} we have that
\begin{align*}
(C\cap F)^\circ=P_F(C^\circ)&=P_F\left(\textrm{conv}\{u_j\,:\,1\leqslant j\leqslant m\}\right)\\
&=\textrm{conv}\{P_F u_j \,:\,1\leqslant j\leqslant m\}.
\end{align*}
Let us denote, for any $y=(x,r)\in H=F\times \R$
$$
N(y)=\inf\left\{\sum_{j=1}^m\frac{\kappa_j\theta_j}{\sqrt{n\Vert P_Fu_j\Vert_2^2+1}}\,:\,\theta_j\geqslant 0,\,y=\sum_{j=1}^m\kappa_j\theta_jw_j\right\},
$$
where the latter infimum is understood as $\infty$ if there do not exist $\{\theta_j\}_{j=1}^m$ with $\theta_j\geqslant 0$ such that $y=\sum_{j=1}^m\kappa_j\theta_jw_j$. Notice that, for any $\{\theta_j\}_{j=1}^m\subseteq\R$,
\begin{eqnarray*}
y=\sum_{j\in J}\kappa_j\theta_jw_j&\Leftrightarrow&(x,r)=\left(-\sum_{j=1}^m\frac{\kappa_j\theta_jP_F(u_j)}{\sqrt{\Vert P_Fu_j\Vert_2^2+\frac{1}{n}}},\sum_{j=1}^m\frac{\kappa_j\theta_j }{\sqrt{n\Vert P_Fu_j\Vert_2^2+1}}\right)\cr
&\Leftrightarrow&(x,r)=\left(-r\sqrt{n}\sum_{j=1}^m\frac{\kappa_j\theta_jP_F(u_j)}{r\sqrt{n\Vert P_Fu_j\Vert_2^2+1}},\sum_{j=1}^m\frac{\kappa_j\theta_j }{\sqrt{n\Vert P_Fu_j\Vert_2^2+1}}\right)\cr
\end{eqnarray*}
and then there exist $\{\theta_j\}_{j=1}^m\subseteq\R$ with $\theta_j\geq0$ for every $1\leqslant j\leqslant m$ such that the latter equality holds if and only if
$$
(x,r)\in L_1:=\left\{(x,r)\in F\times\R\,:\,r\geq0\,:\,x\in -r\sqrt{n}(C\cap F)^\circ\right\},
$$
and for all such $y=(x,r)\in L_1$  we have that $N(y)=r$. Therefore, for every $y\in H$
$$
\sup_{y=\sum_{j=1}^m\kappa_j\theta_j w_j}\prod_{j=1}^m\left(\chi_{[0,\infty)}(\theta_j)e^{-\frac{\theta_j}{\sqrt{n\Vert P_F u_j\Vert_2^2+1}}}\right)^{\kappa_j}=e^{-N(y)}.
$$
Thus, by the reverse Brascamp-Lieb inequality,
$$
\int_{H}e^{-N(y)}dy\geqslant\prod_{j=1}^m\left(\int_0^\infty e^{-\frac{t}{\sqrt{n\Vert P_F u_j\Vert_2^2+1}}}dt\right)^{\kappa_j}.
$$
On the one hand,
$$
\int_{H}e^{-N(y)}dy=\int_0^\infty e^{-r}\left|-r\sqrt{n}(C\cap F)^\circ\right|dr=k!n^{k/2}|(C\cap F)^\circ|.
$$
On the other hand, for every $1\leqslant j\leqslant m$
$$
\int_0^\infty e^{-\frac{t}{\sqrt{n\Vert P_F u_j\Vert_2^2+1}}}dt=\sqrt{n\Vert P_F u_j\Vert_2^2+1}=\sqrt{n+1}\Vert P_Hv_j\Vert_2.
$$
Therefore, since $(K\cap F)^\circ\supseteq (C\cap F)^\circ$, we obtain
\begin{eqnarray*}
|(K\cap F)^\circ|&\geqslant&\frac{(n+1)^\frac{k+1}{2}\prod_{j=1}^m\Vert P_Hv_j\Vert_2^{\delta_j\Vert P_H v_j\Vert_2^2}}{k!n^{k/2}}\cr
&=&\left(\frac{n+1}{k+1}\right)^\frac{k+1}{2}\left(\frac{k}{n}\right)^{k/2}\prod_{j=1}^m\Vert P_Hv_j\Vert_2^{\delta_j\Vert P_H v_j\Vert_2^2}|S_k^\circ|.
\end{eqnarray*}
Since the function $f(x)=x\log x$ is convex in $(0,\infty)$ we have, by Jensen's inequality, that
$$
\sum_{j=1}^m\frac{\delta_j}{n+1}\Vert P_H v_j\Vert^2\log\Vert P_H v_j\Vert_2^2\geq f\left(\sum_{j=1}^m\frac{\delta_j\Vert P_H v_j\Vert_2^2}{n+1}\right)=f\left(\frac{k+1}{n+1}\right)
$$
and then
$$
\prod_{j=1}^m\Vert P_Hv_j\Vert_2^{\delta_j\Vert P_H v_j\Vert_2^2}=e^{\frac{n+1}{2}\sum_{j=1}^m\frac{\delta_j}{n+1}\Vert P_H v_j\Vert^2\log\Vert P_H v_j\Vert_2^2}\geq\left(\frac{k+1}{n+1}\right)^\frac{k+1}{2}.
$$
Therefore,
$$
|(K\cap F)^\circ|^\frac{1}{k}\geqslant \sqrt{\frac{k}{n}}|S_k^\circ|^\frac{1}{k}.
$$
\end{proof}

\section{Mean width of sections of convex bodies in John's position}\label{sec:MeanWidthSections}

In this section we will prove Theorem \ref{thm:MeanWidthJohnSetions}.

\begin{proof}[Proof of Theorem \ref{thm:MeanWidthJohnSetions}.] Let us start with the symmetric case. Assume that $K$ is a centrally symmetric convex body in John's position and $F\in G_{n,k}$ is a $k$-dimensional linear subspace. Following the notation in Section \ref{sec:GeneralSetting} we have that $K\subseteq C_0$, $K\cap F\subseteq C_0\cap F$, and
$$
(K\cap F)^\circ\supseteq(C_0\cap F)^\circ=\textrm{conv}\{\pm P_F u_j \,:\,j\in J_0\}.
$$
Thus, for every $x\in F$,
\begin{eqnarray*}
h_{K\cap F}(x)&\leqslant& h_{C_0\cap F}(x)=\inf\left\{\sum_{j\in J_0}|\alpha_j|\,:\,x=\sum_{j\in J}\alpha_jP_Fu_j\right\}\cr
&=&\inf\left\{\sum_{j\in J_0}|\alpha_j|\,:\,x=\sum_{j\in J_0}\alpha_j\Vert P_Fu_j\Vert_2v_j^0\right\}\cr
&=&\inf\left\{\sum_{j\in J_0}|\beta_j|t_j\,:\,x=\sum_{j\in J_0}\beta_jv_j^0\right\},\cr
\end{eqnarray*}
where for every $j\in J_0$ we have set $t_j=\frac{1}{\Vert P_F u_j\Vert_2}$. Since, for every $x\in F$, we have that
$$
x=\sum_{j\in J_0}\delta_j^0\langle x, v_j^0\rangle v_j^0,
$$
then, for every $\theta\in S_F$,
$$
h_{K\cap F}(x)\leqslant\sum_{j\in J_0}\delta_j^0 t_j|\langle x,v_j^0\rangle|.
$$
Therefore, if $G_1$ is a standard Gaussian random vector in $F$ and $G_2$ is a standard Gaussian random vector on $\R^k$,
\begin{eqnarray*}
\E h_{K\cap F}(G_1)&\leqslant& \sum_{j\in J_0}\delta_j^0 t_j\E|\langle G_1,v_j\rangle|= \E|\langle G_2,e_1\rangle|\sum_{j\in J_0}\delta_j^0 t_j\cr
&=&\frac{1}{k}\E\Vert G_2\Vert_1\sum_{j\in J_0}\delta_j^0 t_j=\frac{1}{k}\sum_{j\in J_0}\delta_j^0 t_j\E h_{B_\infty^k}(G_2).
\end{eqnarray*}
Since, by H\"older's inequality,
\begin{eqnarray*}
\frac{1}{k}\sum_{j\in J_0}\delta_j^0 t_j&=&\frac{1}{k}\sum_{j\in J_0}c_j\Vert P_Fu_j\Vert_2\leqslant\frac{1}{k}\left(\sum_{j\in J_0}c_j\right)^\frac{1}{2}\left(\sum_{j\in J_0}c_j\Vert P_F u_j\Vert_2^2\right)^\frac{1}{2}\cr
&\leqslant& \frac{1}{k}\left(\sum_{j=1}^mc_j\right)^\frac{1}{2}\left(\sum_{j=1}^mc_j\Vert P_F u_j\Vert_2^2\right)^\frac{1}{2}\leqslant\frac{\sqrt{nk}}{k}=\sqrt{\frac{n}{k}},
\end{eqnarray*}
we obtain
$$
\E h_{K\cap F}(G_1)\leqslant\sqrt{\frac{n}{k}}\E h_{B_\infty^k}(G_2).
$$
Equivalently
$$
w(K\cap F)\leqslant \sqrt{\frac{n}{k}}w(B_\infty^k).
$$
Let us now assume that $K$ is a not necessarily symmetric convex body  in John's position and let $F\in G_{n,k}$. Following the notation in Section \ref{sec:GeneralSetting}, we have that for every $x\in F$
$$
h_{K\cap F}(x)\leqslant h_{C\cap F}(x)=\inf\left\{\sum_{j=1}^m a_j:x=\sum_{j=1}^m a_j P_Fu_j, a_j\geqslant 0\right\}.
$$
Let $\theta\in S_F$. Since $\displaystyle{\sum_{j=1}^mc_jP_Fu_j=0}$ we may write
$$\theta=\sum_{j=1}^mc_j\Vert P_F u_j\Vert_2\left (\left\langle
\theta,\frac{P_Fu_j}{\Vert P_F u_j\Vert_2}\right\rangle-\min_{1\leqslant k\leqslant m}\left\langle \theta,\frac{P_Fu_k}{\Vert P_F u_k\Vert_2}\right\rangle\right )P_Fu_j
$$
and then, setting (like in the symmetric case before) $J_0=\{ 1\leqslant j\leqslant m:P_Fu_j\neq 0\}$ and $v_j^0=\frac{P_Fu_j}{\Vert P_Fu_j\Vert_2}$ for $j\in J_0$, we get
\begin{align*}w(K\cap F)\leqslant w(C\cap F) &=\int_{S_F}h_{C\cap
	F}(\theta)d\sigma (\theta)\\
&\leqslant \int_{S_F}\sum_{j=1}^m c_j\Vert P_Fu_j\Vert_2\left (\langle
\theta,v_j^0\rangle-\min_{1\leqslant k\leqslant m}\langle \theta,v_k^0\rangle\right )d\sigma (\theta)\\
&= \int_{S_F}\sum_{j=1}^m c_j\Vert P_Fu_j\Vert_2\max_{1\leqslant k\leqslant m}\left|\langle
\theta,v_k^0\rangle\right|d\sigma (\theta)\\
&=\sum_{j=1}^m c_j\Vert P_Fu_j\Vert_2\int_{S_F}\max_{1\leqslant k\leqslant m}\left|\langle
\theta,v_k^0\rangle\right|d\sigma (\theta).\\
\end{align*}
Since by Cauchy-Schwarz inequality
$$
\sum_{j=1}^m c_j\Vert P_Fu_j\Vert_2\leqslant\left(\sum_{j=1}^m c_j\right)^\frac{1}{2}\left(\sum_{j=1}^mc_j\Vert P_Fu_j\Vert_2^2\right)^\frac{1}{2}=\sqrt{nk}
$$
and since there exists an absolute constant $c>0$ such that
$$
\int_{S_F}\max_{1\leqslant l\leqslant m}\left|\langle \theta,v_l^0\rangle\right|d\sigma (\theta)\leqslant c\sqrt{\frac{\log m}{k}},
$$
taking into account that $m=O(n^2)$ and that $w(S_k)\simeq \sqrt{k\log k}$ we obtain that there exists an absolute constant $C>0$ such that
$$
w(K\cap F)\leqslant C\sqrt{\frac{n\log n}{k\log k}}w(S_k).
$$
\end{proof}


\section{Mean width of projections of convex bodies in L\"owner's position}\label{sec:MeanWidthProjections}

In this section we will prove Theorem \ref{thm:MeanWidthLownerProjections}. We will make use of the following lemma.

\begin{lemma}\label{lem:GaussianMeasure}
Let $K\subseteq\R^n$ be a (not necessarily symmetric) convex body, let $F$ be a $k$-dimensional affine subspace at distance $d$ from the origin and let $\alpha\in\R$ and $\beta\leqslant 0$. Let us identify $F$ with $\R^k$ with the origin at the closest point in $F$ to $0$ and let $\gamma_k$ be the $k$-dimensional Gaussian measure on $F$. Then,
\begin{align*}
\sqrt{\frac{n+d^2}{n}}&\int_0^\infty\frac{e^{-\frac{\left(r-\alpha\sqrt{n+1}\right)^2}{2}}}{\sqrt{2\pi}}e^{\beta r\sqrt{n+1}}\gamma_k\left(\frac{r}{\sqrt{n}}(C\cap F)\right)dr\leqslant\cr
\leqslant&\int_{0}^\infty \frac{e^{-\frac{(r-\alpha d_1)^2}{2}}}{\sqrt{2\pi}}e^{\beta\sqrt{k+1}r}\gamma_k\left(r\sqrt{k+1}\Delta_k\right)dr\cr
\end{align*}
where $d_1=\frac{1}{\sqrt{k+1}}\sum_{j\in J}\delta_j\Vert P_Hv_j\Vert_2$, and $H, J, v_j$, and $\delta_j$ are defined as in Section~\ref{sec:GeneralSetting}.
\end{lemma}
\begin{proof}
Following the notation introduced in Section \ref{sec:GeneralSetting} we have that if $H\in G_{n+1,k+1}$ is the $(k+1)$-dimensional linear subspace $H=\textrm{span}\{(x,\sqrt{n})\,:\,x\in F\}$ then
$$
L\cap H=\Big\{(x,r)\,:\,r\geqslant0,\,x\in\frac{r}{\sqrt{n}}(C\cap F)\Big\}.
$$
For any $\alpha,\beta\in\R$, let $\mu_{\alpha,\beta}$ be the measure on $H$ whose density with respect to the Lebesgue measure at a point $y=(x,r)$ is
$$
d\mu_{\alpha,\beta}(r)=\frac{e^{-\frac{\Vert y\Vert_2^2}{2}}}{(2\pi)^{\frac{k+1}{2}}}e^{(\alpha+\beta)\sqrt{n+1}r}dy.
$$
Denoting by $\gamma_k$ the $k$-dimensional Gaussian measure on $F$, we have that, on the one-hand, for any $\alpha,\beta\in\R$
\begin{eqnarray*}
\mu_{\alpha,\beta}(L\cap H)&=&\int_0^\infty\frac{e^{-\frac{r^2}{2}}}{\sqrt{2\pi}}e^{\alpha r\sqrt{n+1}}e^{\beta\sqrt{n+1} r}\gamma_k\left(\frac{r}{\sqrt{n}}(C\cap F)\right)\sqrt{\frac{n+d^2}{n}}dr\cr
&=&e^{\frac{\alpha^2(n+1)}{2}}\int_0^\infty\frac{e^{-\frac{\left(r-\alpha\sqrt{n+1}\right)^2}{2}}}{\sqrt{2\pi}}e^{\beta\sqrt{n+1}r}\gamma_k\left(\frac{r}{\sqrt{n}}(C\cap F)\right)\sqrt{\frac{n+d^2}{n}}dr.
\end{eqnarray*}
On the other hand, setting $J=\{1\leqslant j\leqslant m\,:\, P_Hv_j\neq 0\}$ and  $s_j=\frac{1}{\Vert P_Hv_j\Vert_2}$ for every $j\in J$, and denoting $y=(x,r)\in H$ we have, by the Brascamp-Lieb inequality,
\begin{eqnarray*}
&&\mu_{\alpha,\beta}(L\cap H)=\int_{H}\frac{e^{-\frac{\Vert y\Vert_2^2}{2}}}{(2\pi)^{(k+1)/2}}e^{\alpha\sqrt{n+1}r}e^{\beta\sqrt{n+1} r}\prod_{j\in J}\chi_{[0,\infty)}(\langle y,v_j\rangle)dy\cr
&=&\int_{H}\frac{e^{-\frac{\Vert y\Vert_2^2}{2}}}{(2\pi)^{(k+1)/2}}e^{\alpha \sqrt{n+1}r}e^{\beta \sqrt{n+1}r}\prod_{j\in J}\chi_{[0,\infty)}(\langle y,w_j\rangle)dy\cr
&=&\int_{H}\frac{e^{-\frac{\sum_{j\in J}\kappa_j\langle y,w_j\rangle^2}{2}}}{(2\pi)^{(k+1)/2}}e^{\sum_{j\in J}\kappa_j (\alpha+\beta) s_j\langle y,w_j\rangle}\prod_{j\in J}\chi_{[0,\infty)}(\langle y,w_j\rangle)dy\cr
&=&\int_{H}\prod_{j\in J}\left(\frac{e^{-\frac{\langle y,w_j\rangle^2}{2}}}{\sqrt{2\pi}}e^{(\alpha+\beta)s_j\langle y,w_j\rangle}\chi_{[0,\infty)}(\langle y,w_j\rangle)\right)^{\kappa_j}dy\cr
&\leqslant&\prod_{j\in J}\left(\int_0^\infty\frac{e^{-\frac{t^2}{2}}}{\sqrt{2\pi}}e^{\alpha s_jt}e^{\beta s_jt}dt\right)^{\kappa_j}=e^{\frac{\alpha^2\sum_{j\in J}\kappa_js_j^2}{2}}\prod_{j=1}^m\left(\int_0^\infty\frac{e^{-\frac{\left(t-\alpha s_j\right)^2}{2}}}{\sqrt{2\pi}}e^{\beta s_j t}dt\right)^{\kappa_j}\cr
&=&e^{\frac{\alpha^2(n+1)}{2}}\prod_{j\in J}\left(\int_0^\infty\frac{e^{-\frac{\left(t-\alpha s_j\right)^2}{2}}}{\sqrt{2\pi}}e^{\beta s_j t}dt\right)^{\kappa_j}\cr
\end{eqnarray*}
Therefore, for any $\alpha,\beta\in\R$,
\begin{align*}
&\sqrt{\frac{n+d^2}{n}}\int_0^\infty\frac{e^{-\frac{\left(r-\alpha\sqrt{n+1}\right)^2}{2}}}{\sqrt{2\pi}}e^{\beta r\sqrt{n+1}}\gamma_k\left(\frac{r}{\sqrt{n}}(C\cap F)\right)dr\\
&\hspace*{2cm}\leqslant\prod_{j\in J}\left(\int_0^\infty\frac{e^{-\frac{\left(t-\alpha s_j\right)^2}{2}}}{\sqrt{2\pi}}e^{\beta s_jt}dt\right)^{\kappa_j},
\end{align*}
and, by the Pr\'ekopa-Leindler inequality, setting $d_1=\frac{1}{\sqrt{k+1}}\sum_{j=1}^m\kappa_j s_j$ we have that if $\beta\leqslant0$ then (taking into account that $s_j=\frac{1}{\Vert P_H v_j\Vert_2}\geqslant 1$)
\begin{eqnarray*}
&&\sqrt{\frac{n+d^2}{n}}\int_0^\infty\frac{e^{-\frac{\left(r-\alpha\sqrt{n+1}\right)^2}{2}}}{\sqrt{2\pi}}e^{\beta r\sqrt{n+1}}\gamma_k\left(\frac{r}{\sqrt{n}}(C\cap F)\right)dr\cr
&\leqslant&\prod_{j\in J}\left(\int_0^\infty\frac{e^{-\frac{\left(t-\alpha s_j\right)^2}{2}}}{\sqrt{2\pi}}e^{\beta s_jt}dt\right)^{\kappa_j}\cr
&\leqslant&\prod_{j\in J}\left(\int_0^\infty\frac{e^{-\frac{\left(t-\alpha s_j\right)^2}{2}}}{\sqrt{2\pi}}e^{\beta t}dt\right)^{\kappa_j}\leqslant\left(\int_0^\infty \frac{e^{-\frac{\left(t-\frac{\alpha}{k+1}\sum_{j\in J}\kappa_j s_j\right)^2}{2}}}{\sqrt{2\pi}}e^{\beta t}dt\right)^{k+1}\cr
&=&\int_{[0,\infty)^{k+1}} \prod_{i=1}^{k+1}\frac{e^{-\frac{\left(t_i-\frac{\alpha d_1}{\sqrt{k+1}}\right)^2}{2}}}{\sqrt{2\pi}}e^{\beta t_i}dt\cr
&=&\int_{\left[-\frac{\alpha d_1}{\sqrt{k+1}},\infty\right)^{k+1}} \prod_{i=1}^{k+1}\frac{e^{-\frac{\Vert t\Vert_2^2}{2}}}{\sqrt{2\pi}}e^{\beta\sqrt{k+1}\langle t,v_0\rangle}e^{\beta \alpha d_1\sqrt{k+1}}dt,\cr
\end{eqnarray*}
where $v_0=\left(\frac{1}{\sqrt{k+1}},\dots,\frac{1}{\sqrt{k+1}}\right)$.
Therefore, for any $\alpha\in\R$ and any $\beta\leqslant0$,
\begin{eqnarray*}
&&\sqrt{\frac{n+d^2}{n}}\int_0^\infty\frac{e^{-\frac{\left(r-\alpha\sqrt{n+1}\right)^2}{2}}}{\sqrt{2\pi}}e^{\beta r\sqrt{n+1}}\gamma_k\left(\frac{r}{\sqrt{n}}(C\cap F)\right)dr\cr
&\leqslant&\int_{\left[-\frac{\alpha d_1}{\sqrt{k+1}},\infty\right)^{k+1}} \prod_{i=1}^{k+1}\frac{e^{-\frac{\Vert t\Vert_2^2}{2}}}{\sqrt{2\pi}}e^{\beta\sqrt{k+1}\langle t,v_0\rangle}e^{\beta \alpha d_1\sqrt{k+1}}dt\cr
&=&\int_{-\alpha d_1}^\infty \frac{e^{-\frac{t^2}{2}}}{\sqrt{2\pi}}e^{\beta\sqrt{k+1}t}e^{\beta \alpha d_1\sqrt{k+1}}\gamma_k\left((t+\alpha d_1)\sqrt{k+1}\Delta_k\right)dt\cr
&=&\int_{0}^\infty \frac{e^{-\frac{(r-\alpha d_1)^2}{2}}}{\sqrt{2\pi}}e^{\beta\sqrt{k+1}r}\gamma_k\left(r\sqrt{k+1}\Delta_k\right)dr.
\end{eqnarray*}
\end{proof}

\begin{proof}[Proof of Theorem \ref{thm:MeanWidthLownerProjections}] Let us start with the symmetric case. Assume that $K$ is a centrally symmetric convex body in John's position and $F\in G_{n,k}$ is a $k$-dimensional linear subspace.
We want to prove that
$$
w((K\cap F)^\circ)\geqslant w\left(\left(\sqrt{\frac{n}{k}}B_\infty^k\right)^{\circ}\right),
$$
which is equivalent to
$$
\E\Vert G_1\Vert_{K\cap F}\geqslant\E\Vert G_2\Vert_{\sqrt{\frac{n}{k}}B_\infty^k},
$$
where $G_1$ is a standard Gaussian random vector on $F$ and $G_2$ is a standard Gaussian random vector on $\R^k$.
Since for any convex body $L\subseteq\R^k$ containing the origin in its interior we have, by Fubini's theorem, that if $G$ is a standard Gaussian random vector then
$$
\E\Vert G\Vert_L=\int_0^\infty\Pro (\Vert G\Vert_L\geqslant t)dt=\int_0^\infty\gamma_k(\R^n\setminus tL)dt
$$
where $\gamma_k(A)$ denotes the Gaussian measure of the $k$-dimensional set $A$, we obtain that the statement we want to prove is equivalent to
$$
\int_0^\infty\gamma_k(F\setminus t(K\cap F))dt\geqslant \int_0^\infty\gamma_k\left(\R^n\setminus t\sqrt{\frac{n}{k}}B_\infty^k\right)dt
$$
or, equivalently,
$$
\int_0^\infty (1-\gamma_k(t(K\cap F))dt\geqslant \int_0^\infty\left(1-\gamma_k\left(t\sqrt{\frac{n}{k}}B_\infty^k\right)\right)dt.
$$
We are going to prove that for any $t\geqslant 0$
$$
\gamma_k(t(K\cap F))\leqslant\gamma_k\left(t\sqrt{\frac{n}{k}} B_\infty^k\right),
$$
which implies the latter inequality.

Following the notation in Section \ref{sec:GeneralSetting}, since $K$ is a centrally symmetric convex body in John's position we have that
$$
K\cap F\subseteq C_0\cap F=\left\{x\in F\,:\, |\langle x, v_j^0\rangle|\leqslant t_j,\,\forall j\in J_0\right\},
$$
where, for every $j\in J_0$ we have set $t_j=\frac{1}{\Vert P_F u_j\Vert_2}=\left(\frac{c_j}{\delta_j^0}\right)^{1/2}$. Therefore, for every $t\geqslant 0$,
$$
t(K\cap F)\subseteq t(C_0\cap F)=\left\{x\in F\,:\, |\langle x, v_j^0\rangle|\leqslant t t_j,\,\forall j\in J\right\},
$$
and then, by the Brascamp-Lieb inequality
\begin{eqnarray*}
\gamma_k(t(K\cap F))&\leqslant&\gamma_k(t(C_0\cap F))=\int_{F}\left(\prod_{j\in J}\chi_{[-tt_j,tt_j]}(\langle x, v_j^0\rangle) \right)\frac{e^{-\frac{\Vert x\Vert_2^2}{2}}}{(2\pi)^{k/2}}dx\cr
&=&\int_{F}\left(\prod_{j\in J_0}\chi_{[-tt_j,tt_j]}(\langle x, v_j^0\rangle) \right)\frac{e^{-\frac{\sum_{j\in J_0}\delta_j^0\langle x,v_j^0\rangle^2}{2}}}{(2\pi)^{k/2}}dx\cr
&=&\int_{F}\prod_{j\in J_0}\left(\chi_{[-tt_j,tt_j]}(\langle x, v_j^0\rangle) \frac{e^{-\frac{\langle x,v_j^0\rangle^2}{2}}}{\sqrt{2\pi}}\right)^{\delta_j^0}dx\cr
&\leqslant&\prod_{j\in J_0}\left(\int_{-tt_j}^{tt_j}\frac{e^{-\frac{t^2}{2}}}{\sqrt{2\pi}}dt\right)^{\delta_j^0}=\left(\prod_{j\in J_0}\gamma_1(tt_j[-e_1,e_1])^\frac{\delta_j^0}{k}\right)^k.
\end{eqnarray*}
Since $\gamma_1$ is log-concave, we have that
$$
\gamma_k(t(K\cap F))\leqslant\gamma_1\left(\left(t\sum_{j\in J_0}\frac{t_j\delta_j^0}{k}\right)[-e_1,e_1]\right)^k=\gamma_k\left(\left(t\sum_{j\in J_0}\frac{t_j\delta_j^0}{k}\right) B_\infty^k\right).
$$
Since, by H\"older's inequality
\begin{eqnarray*}
\sum_{j\in J_0}\frac{t_j\delta_j^0}{k}&=&\sum_{j\in J_0}\frac{\sqrt{c_j\delta_j^0}}{k}\leqslant\frac{1}{k}\left(\sum_{j\in J_0} c_j\right)^{1/2}\left(\sum_{j\in J_0} \delta_j^0\right)^{1/2}\cr
&\leqslant&\frac{1}{k}\left(\sum_{j=1}^m c_j\right)^{1/2}\left(\sum_{j=1}^m \delta_j^0\right)^{1/2}=\sqrt{\frac{n}{k}},
\end{eqnarray*}
we obtain that for every $t\geqslant0$
$$
\gamma_k(t(K\cap F))\leqslant\gamma_k\left(t\sqrt{\frac{n}{k}}B_\infty^k\right).
$$

Assume now that $K\subseteq\R^n$ is a (not necessarily symmetric) convex body in John's position and $F\in G_{n,k}$ is a linear subspace. Following the notation introduced in Section \ref{sec:GeneralSetting} we have that if $H=F\times\R$ then
$$
L\cap H=\left\{(x,r)\in F\times\R\,:\,r\geqslant 0,\, x\in \frac{r}{\sqrt{n}}(C\cap F)\right\}.
$$

By Lemma \ref{lem:GaussianMeasure} with $\beta=0$ and taking into account that $F\in G_{n,k}$ is a linear subspace, for any $\alpha\in\R$
we have that
$$
\int_0^\infty\frac{e^{-\frac{\left(r-\alpha\sqrt{n+1}\right)^2}{2}}}{\sqrt{2\pi}}\gamma_k\left(\frac{r}{\sqrt{n}}(C\cap F)\right)dr\leqslant \int_0^\infty\frac{e^{-\frac{\left(r-\alpha d_1\right)^2}{2}}}{\sqrt{2\pi}}\gamma_k\left(r\sqrt{k+1}\Delta_{k}\right)dr,
$$
where $d_1=\frac{1}{\sqrt{k+1}}\sum_{j\in J}\delta_j\Vert P_Hv_j\Vert_2$ . Applying the latter inequality to $-\alpha$, we also get
$$
\int_0^\infty\frac{e^{-\frac{\left(r+\alpha\sqrt{n+1}\right)^2}{2}}}{\sqrt{2\pi}}\gamma_k\left(\frac{r}{\sqrt{n}}(C\cap F)\right)dr\leqslant \int_0^\infty\frac{e^{-\frac{\left(r+\alpha d_1\right)^2}{2}}}{\sqrt{2\pi}}\gamma_k\left(r\sqrt{k+1}\Delta_{k}\right)dr
$$
or, equivalently,
$$
\int_{-\infty}^0\frac{e^{-\frac{\left(r-\alpha\sqrt{n+1}\right)^2}{2}}}{\sqrt{2\pi}}\gamma_k\left(\frac{|r|}{\sqrt{n}}(C\cap F)\right)dr\leqslant \int_{-\infty}^0\frac{e^{-\frac{\left(r-\alpha d_1\right)^2}{2}}}{\sqrt{2\pi}}\gamma_k\left(|r|\sqrt{k+1}\Delta_{k}\right)dr.
$$
Therefore, for any $\alpha\in\R$,
$$
\int_{-\infty}^\infty\frac{e^{-\frac{\left(r-\alpha\sqrt{n+1}\right)^2}{2}}}{\sqrt{2\pi}}\gamma_k\left(\frac{|r|}{\sqrt{n}}(C\cap F)\right)dr\leqslant \int_{-\infty}^\infty\frac{e^{-\frac{\left(r-\alpha d_1\right)^2}{2}}}{\sqrt{2\pi}}\gamma_k\left(|r|\sqrt{k+1}\Delta_{k}\right)dr
$$
and hence,
$$
\int_{-\infty}^\infty\frac{e^{-\frac{\left(r-\alpha\sqrt{n+1}\right)^2}{2}}}{\sqrt{2\pi}}\gamma_k(F\setminus\left(\left(\frac{|r|}{\sqrt{n}}(C\cap F)\right)\right)dr\geqslant \int_{-\infty}^\infty\frac{e^{-\frac{\left(r-\alpha d_1\right)^2}{2}}}{\sqrt{2\pi}}\gamma_k\left(\R^k\setminus(|r|\sqrt{k+1}\Delta_{k})\right)dr
$$
Integrating in $\alpha\in\R$ we obtain
$$
\frac{1}{\sqrt{n+1}}\int_{-\infty}^\infty\gamma_k\left(F\setminus\left(\frac{|r|}{\sqrt{n}}(C\cap F)\right)\right)dr\geqslant \frac{1}{d_1}\int_{-\infty}^\infty\gamma_k\left(\R^k\setminus(|r|\sqrt{k+1}\Delta_{k})\right)dr.
$$
Equivalently,
$$
\frac{1}{\sqrt{n+1}}\int_{0}^\infty\gamma_k\left(F\setminus\left(\frac{r}{\sqrt{n}}(C\cap F)\right)\right)dr\geqslant \frac{1}{d_1}\int_{0}^\infty\gamma_k\left(\R^k\setminus(r\sqrt{k+1}\Delta_{k})\right)dr,
$$
or
$$
\sqrt{\frac{n}{n+1}}\int_{0}^\infty\gamma_k\left(F\setminus\left(r(C\cap F)\right)\right)dr\geqslant \frac{1}{d_1\sqrt{k+1}}\int_{0}^\infty\gamma_k\left(\R^k\setminus(r\Delta_{k})\right)dr.
$$
Integrating in polar coordinates, and taking into account that $K\subseteq C$ we obtain
$$
w((K\cap F)^\circ)\geqslant\frac{1}{d_1}\sqrt{\frac{n+1}{n(k+1)}}w((\Delta_k)^\circ).
$$
Since
$$
\sqrt{k(k+1)}\Delta_k=S_k,
$$
where $S_k$ denotes the $k$-dimensional regular simplex in John's position, we have that, for any $k$-dimensional subspace $F$,
$$
w((K\cap F)^\circ)\geqslant\frac{1}{d_1}\sqrt{\frac{k(n+1)}{n}}w((S_k)^\circ).
$$
Thus, for any $k$-dimensional subspace $F$,
$$
w((K\cap F)^\circ)\geqslant\sqrt{\frac{k(k+1)(n+1)}{n}}\frac{1}{d_1}w((S_k)^\circ).
$$
Since
$$
d_1=\sum_{j=1}^m\delta_j\Vert P_Hv_j\Vert_2\leqslant\left(\sum_{j=1}^m\delta_j\right)^\frac{1}{2}\left(\sum_{j=1}^m\delta_j\Vert P_Hv_j\Vert_2^2\right)^\frac{1}{2}=\sqrt{(n+1)(k+1)},
$$
we have that
$$
w((K\cap F)^\circ)\geqslant\sqrt{\frac{k}{n}}w(S_k^\circ).
$$
\end{proof}

\section{The Wills functional of sections of convex bodies in John's position}\label{sec:WillsFunctionalSections}

In this section we will give the proof of Theorem \ref{thm:WillsJohnSymmetric}.

\begin{proof}[Proof of Theorem \ref{thm:WillsJohnSymmetric}.]
Let $K$ be a centrally symmetric convex body in John's position and $F\in G_{n,k}$ a $k$-dimensional linear subspace. Following the notation in Section \ref{sec:GeneralSetting} we have that, for every $\lambda\geqslant0$,
$$
\lambda(K\cap F)\subseteq \lambda(C_0\cap F)=\left\{x\in F\,:\, |\langle x, v_j^0\rangle|\leqslant \lambda t_j,\,\forall j\in J_0\right\},
$$
where, for every $j\in J_0$, we denote $t_j=\frac{1}{\Vert P_F u_j\Vert_2}=\left(\frac{c_j}{\delta_j^0}\right)^{1/2}$.
Let, for every $j\in J_0$, $f_j:\R\to[0,\infty)$ be the function
$$
f_j(t)=e^{-\pi d(tv_j^0,P_{\langle v_j^0\rangle}(\lambda(C_0\cap F)))^2}\quad\forall t\in\R,
$$
where $\langle v_j^0\rangle$ denotes the $1$-dimensional subspace spanned by $v_j^0$. Notice that, for every $j\in J_0$,
$$
\int_{\R}f_j(t)dt=\int_{\R}e^{-\pi d(tv_j^0,P_{\langle v_j^0\rangle}(\lambda(C_0\cap F)))^2}dt=\mathcal{W}(P_{\langle v_j^0\rangle}(\lambda(C_0\cap F))),
$$
and, since for every $j\in J_0$ we have that
$$
P_{\langle v_j^0\rangle}(\lambda(C_0\cap F))\subseteq\left[-\lambda t_j,\lambda t_j\right]v_j,
$$
we see that, for every $j\in J_0$,
$$
\int_{\R}f_j(t)dt\leqslant \mathcal{W}\left(\left[-\lambda t_j,\lambda t_j\right]v_j\right)=\left(1+2\lambda t_j\right).
$$
Therefore, by the Brascamp-Lieb inequality,
\begin{align*}
\int_{F}e^{-\pi\sum_{j\in J_0}\delta_j^0d\left(\langle x,v_j^0\rangle  v_j^0,P_{\langle v_j^0\rangle}(\lambda(C_0\cap F)) \right)^2}
&=\int_{F}\prod_{j\in J_0}f_j^{\delta_j^0}(\langle x,v_j^0\rangle) dx\\
&\leqslant\prod_{j\in J_0}\left(\int_{\R}f_j(t)dt\right)^{\delta_j^0}\\
&\leqslant \prod_{j\in J_0}\left(1+2\lambda t_j\right)^{\delta_j^0}.
\end{align*}
By the arithmetic-geometric mean inequality we have
\begin{eqnarray*}
\prod_{j\in J_0}\left(1+2\lambda t_j\right)^\frac{\delta_j^0}{k}&\leqslant&\sum_{j\in J_0}\frac{\delta_j^0}{k}\left(1+2\lambda t_j\right)\leqslant1+\frac{2\lambda}{k}\sum_{j\in J_0}c_j\Vert P_F u_j\Vert_2\cr
&\leqslant&1+\frac{2\lambda}{k}\left(\sum_{j\in J_0}c_j\right)^\frac{1}{2}\left(\sum_{j\in J_0}c_j\Vert P_F u_j\Vert_2^2\right)^\frac{1}{2}\cr
&\leqslant&1+\frac{2\lambda}{k}\left(\sum_{j=1}^mc_j\right)^\frac{1}{2}\left(\sum_{j=1}^mc_j\Vert P_F u_j\Vert_2^2\right)^\frac{1}{2}\cr
&=&1+2\lambda\sqrt\frac{n}{k},
\end{eqnarray*}
and then
$$
\prod_{j\in J_0}\left(1+2\lambda t_j\right)^{\delta_j^0}\leqslant\left(1+2\lambda \sqrt\frac{n}{k}\right)^k=\mathcal{W}\left(\lambda\sqrt\frac{n}{k}B_\infty^k\right).
$$
On the other hand, for every $x_0\in \lambda(C_0\cap F)$ we have that for every $x\in F$ and every $j\in J_0$
$$
d\left(\langle x,v_j^0\rangle  v_j^0,P_{\langle v_j^0\rangle}(\lambda(C_0\cap F)) \right)^2\leqslant d\left(\langle x,v_j^0\rangle  v_j^0,\langle x_0,v_j^0\rangle v_j^0\right)^2=\langle x-x_0,v_j\rangle^2.
$$
Thus, for every $x_0\in \lambda(C_0\cap F)$ and every $x\in F$,
$$
\sum_{j\in J_0}\delta_j^0 d\left(\langle x,v_j^0\rangle  v_j^0,P_{\langle v_j^0\rangle}(\lambda(C_0\cap F)) \right)^2\leqslant\sum_{j\in J_0}\delta_j^0\langle x-x_0,v_j^0\rangle^2=|x-x_0|^2,
$$
and hence, for every $x\in F$,
$$
\sum_{j\in J_0}\delta_j^0 d\left(\langle x,v_j^0\rangle  v_j^0,P_{\langle v_j^0\rangle}(\lambda(C_0\cap F)) \right)^2\leqslant d(x,\lambda(C_0\cap F))^2.
$$
Consequently,
\begin{eqnarray*}
\mathcal{W}(\lambda(C_0\cap F))&=&\int_{F}e^{-\pi d(x,\lambda(C_0\cap F))^2}dx\leqslant\int_{F}e^{-\pi\sum_{j\in J_0}\delta_j^0 d\left(\langle x,v_j^0\rangle  v_j^0,P_{\langle v_j^0\rangle}(\lambda(C_0\cap F)) \right)^2}dx\cr
&\leqslant&\mathcal{W}\left(\lambda\sqrt\frac{n}{k}B_\infty^k\right).
\end{eqnarray*}
Therefore, since $K\cap F\subseteq C_0\cap F$, by the monotonicity of the Wills functional we get
$$
\mathcal{W}(\lambda(K\cap F))\leqslant \mathcal{W}(\lambda(C_0\cap F))\leqslant \mathcal{W}\left(\lambda\sqrt\frac{n}{k}B_\infty^k\right).
$$
\end{proof}

The following result gives a similar upper bound for a quantity defined via a double polarity, both on the convex body and on the log-concave function. We will denote, for any $k$ dimensional linear subspace $F\in G_{n,k}$ and any convex body $K\subseteq F$,
$$
f_K(x)=e^{-\pi d^2(x,K)}\quad\forall x\in F.
$$
\begin{thm}\label{thm:Wills2JohnSections}

Let $K\subseteq\R^n$ be a convex body in John's position and let $F$ be a $k$-dimensional affine subspace at distance $d$ from $0$.
Then, for every $\lambda>0$,
$$
\int_{F}f_{(\lambda(K\cap F))^\circ}^\circ(x)dx\leqslant\frac{n+1}{k+1}\sqrt{\frac{n}{n+d^2}}\int_{\R^k}f_{\left(\lambda\sqrt{\frac{n(n+1)}{k(k+1)}}S_k\right)^\circ}^\circ(x)dx,
$$
where the polarity is taken with respect to the closest point in $F$ to the origin if it belongs to the relative interior to $K\cap F$.

Furthermore, if $K$ is centrally symmetric and $F\in G_{n,k}$ is a $k$-dimensional linear subspace then, for every $\lambda>0$,
$$
\int_{F}f_{(\lambda(K\cap F))^\circ}^\circ(x)dx\leqslant\int_{\R^k}f_{\left(\lambda\sqrt{\frac{n}{k}}B_\infty^k\right)^\circ}^\circ(x)dx.
$$
\end{thm}

\begin{proof}
Let $K\subseteq\R^n$ be a centrally symmetric convex body in John's position and let $F\in G_{n,k}$ be a $k$-dimensional linear subspace. From the definition of $f_{(\lambda(K\cap F))^\circ}^\circ$ and \eqref{eq: PolarFunctionWills} we have that, for every $\lambda>0$,
\begin{eqnarray*}
\int_Ff_{(\lambda(K\cap F))^\circ}^\circ(x)dx&=&\int_Fe^{-\frac{\Vert x\Vert_2^2}{4\pi}}e^{-\Vert x\Vert_{\lambda(K\cap F)}}dx=\int_Fe^{-\frac{\Vert x\Vert_2^2}{4\pi}}\int_{\Vert x\Vert_{\lambda(K\cap F)}}^\infty e^{-t}dtdx\cr
&=&\int_0^\infty e^{-t}\int_{t\lambda(K\cap F)}e^{-\frac{\Vert x\Vert_2^2}{4\pi}}dtdx\cr
&=&(2\pi)^k\int_0^\infty e^{-t}\int_{\frac{t\lambda}{\sqrt{2\pi}}(K\cap F)}\frac{e^{-\frac{\Vert x\Vert_2^2}{2}}}{(\sqrt{2\pi})^k}dxdt\cr
&=&(2\pi)^k\int_0^\infty e^{-t}\gamma_k\left(\frac{t\lambda}{\sqrt{2\pi}}(K\cap F)\right)dt.
\end{eqnarray*}
Similarly, for every $\lambda>0$,
$$
\int_{\R^k}f_{\left(\lambda\sqrt{\frac{n}{k}}B_\infty^k\right)^\circ}^\circ(x)dx=(2\pi)^k\int_0^\infty e^{-t}\gamma_k\left(\frac{t\lambda}{\sqrt{2\pi}}\sqrt{\frac{n}{k}}B_\infty^k\right)dt.
$$
As we have seen in the proof of Theorem \ref{thm:MeanWidthLownerProjections}, for every $t\geqslant0$ and every $\lambda>0$,
$$
\gamma_k\left(\frac{t\lambda}{\sqrt{2\pi}}(K\cap F)\right)\leqslant\gamma_k\left(\frac{t\lambda}{\sqrt{2\pi}}\sqrt{\frac{n}{k}}B_\infty^k\right).
$$
Therefore, for every $\lambda>0$,
$$
\int_Ff_{(\lambda(K\cap F))^\circ}^\circ(x)dx\leqslant\int_{\R^k}f_{\left(\lambda\sqrt{\frac{n}{k}}B_\infty^k\right)^\circ}^\circ(x)dx.
$$
Assume now that $K\subseteq\R^n$ is a (not necessarily symmetric) convex body in John's position and $F$ is a $k$-dimensional affine subspace at distance $d$ from the origin. As before, we have that, for every $\lambda>0$,
$$
\int_Ff_{(\lambda(K\cap F))^\circ}^\circ(x)dx=(2\pi)^k\int_0^\infty e^{-t}\gamma_k\left(\frac{t\lambda}{\sqrt{2\pi}}(K\cap F)\right)dt,
$$
where the identity holds identifying the point with respect to which the polar body is taken with the origin in the affine subspace. Following the notation in Section \ref{sec:GeneralSetting}, by Lemma \ref{lem:GaussianMeasure} we have that, for any $\alpha\in\R$ and any $\beta\leqslant0$,
\begin{align*}
\sqrt{\frac{n+d^2}{n}}&\int_0^\infty\frac{e^{-\frac{\left(r-\alpha\sqrt{n+1}\right)^2}{2}}}{\sqrt{2\pi}}e^{\beta r\sqrt{n+1}}\gamma_k\left(\frac{r}{\sqrt{n}}(C\cap F)\right)dr\leqslant\cr
&\int_{0}^\infty \frac{e^{-\frac{(r-\alpha d_1)^2}{2}}}{\sqrt{2\pi}}e^{\beta\sqrt{k+1}r}\gamma_k\left(r\sqrt{k+1}\Delta_k\right)dr,
\end{align*}
where $d_1=\frac{1}{\sqrt{k+1}}\sum_{j\in J}\delta_j\Vert P_Hv_j\Vert_2$. Integrating with respect to $\alpha\in\R$ we see that, for any $\beta\leqslant 0$,
$$
\sqrt{\frac{n+d^2}{n(n+1)}}\int_0^\infty e^{\beta r\sqrt{n+1}}\gamma_k\left(\frac{r}{\sqrt{n}}(C\cap F)\right)dr\leqslant\frac{1}{d_1}\int_{0}^\infty e^{\beta\sqrt{k+1}r}\gamma_k\left(r\sqrt{k+1}\Delta_k\right)dr,
$$
or, equivalently, for any $\beta\leqslant0$,
$$
\sqrt{\frac{n+d^2}{n+1}}\int_0^\infty e^{\frac{\beta u\sqrt{n(n+1)}}{\sqrt{2\pi}}}\gamma_k\left(\frac{u}{\sqrt{2\pi}}(C\cap F)\right)du\leqslant\frac{1}{d_1\sqrt{k+1}}\int_{0}^\infty e^{\frac{\beta u}{\sqrt{2\pi}}}\gamma_k\left(\frac{u}{\sqrt{2\pi}}\Delta_k\right)du.
$$
Taking, for any $\lambda>0$, $\beta=-\frac{1}{\lambda}\sqrt{\frac{2\pi}{n(n+1)}}$ we obtain
$$
\sqrt{\frac{n+d^2}{n+1}}\int_0^\infty e^{-\frac{u}{\lambda}}\gamma_k\left(\frac{u}{\sqrt{2\pi}}(C\cap F)\right)du\leqslant\frac{1}{d_1\sqrt{k+1}}\int_{0}^\infty e^{\frac{-u}{\lambda\sqrt{n(n+1)}}}\gamma_k\left(\frac{u}{\sqrt{2\pi}}\Delta_k\right)du,
$$
or, equivalently,
$$
\sqrt{\frac{n+d^2}{n+1}}\int_0^\infty e^{-u}\gamma_k\left(\frac{u\lambda}{\sqrt{2\pi}}(C\cap F)\right)du\leqslant\frac{\sqrt{n(n+1)}}{d_1s\sqrt{k+1}}\int_{0}^\infty e^{-u}\gamma_k\left(\frac{u\lambda\sqrt{n(n+1)}}{\sqrt{2\pi}}\Delta_k\right)du.
$$
Since $\sqrt{k(k+1)}\Delta_k=S_k$ we see that, for every $\lambda>0$,
$$
\int_0^\infty e^{-u}\gamma_k\left(\frac{u\lambda}{\sqrt{2\pi}}(C\cap F)\right)du\leqslant \frac{(n+1)\sqrt{n}}{d_1\sqrt{(k+1)(n+d^2)}}\int_{0}^\infty e^{-u}\gamma_k\left(\frac{u\lambda}{\sqrt{2\pi}}\sqrt{\frac{n(n+1)}{k(k+1)}}S_k\right)du.
$$
Consequently, for any $\lambda>0$, taking polars with respect to the closest point in $F$ to the origin which belongs to the relative interior of $K\cap F$,
\begin{align*}
\int_F f^\circ_{(\lambda(K\cap F))^\circ}(x)dx &\leqslant \int_F f^\circ_{(\lambda(C\cap F))^\circ}(x)dx\\
&\leqslant \frac{\sqrt{n}(n+1)}{d_1\sqrt{(k+1)(n+1)}}\int_{\R^k} f^\circ_{\left(\lambda\sqrt{\frac{n(n+1)}{k(k+1)}}S_k\right)^\circ}(x)dx.
\end{align*}
Since
$$
d_1\sqrt{k+1}=\sum_{j=1}^m\kappa_j s_j\geqslant\sum_{j=1}^m\kappa_j=k+1,
$$
we have that, for every $\lambda>0$,
$$
\int_F f^\circ_{(\lambda(K\cap F))^\circ}(x)dx\leqslant \frac{n+1}{k+1}\sqrt{\frac{n}{n+d^2}}\int_{\R^k} f^\circ_{\left(\lambda\sqrt{\frac{n(n+1)}{k(k+1)}}S_k\right)^\circ}(x)dx.
$$
\end{proof}

\section{The Wills functional of projections of convex bodies in L\"owner's position}\label{sec:WillsFunctionalProjections}

In this section we will give the proof of Theorem \ref{thm:WillsLownerSymmetric}.
\begin{proof}[Proof of Theorem \ref{thm:WillsLownerSymmetric}.]
Let $K$ be a centrally symmetric convex body in John's position and let $F\in G_{n,k}$ be a $k$-dimensional linear subspace. Following the notation in Section \ref{sec:GeneralSetting} we have that $K\subseteq C_0$, $K\cap F\subseteq C_0\cap F$, and
$$
(K\cap F)^\circ\supseteq (C_0\cap F)^\circ=\textrm{conv}\{\pm\Vert P_F u_j\Vert_2 v_j\,j\in J_0\}.
$$
Therefore, since the function $d(\cdot,(C_0\cap F)^\circ)^2$ is convex, for any $x\in F$ and any $\{\theta_j\}_{j\in J_0}\subseteq\R$ such that $\displaystyle{x=\sum_{j\in J_0}\delta_j^0\theta_jv_j^0}$, we have that
\begin{eqnarray*}
d(x,(K\cap F)^\circ)^2&\leqslant&d(x,(C_0\cap F)^\circ)^2=d\left(\sum_{j\in J_0}\frac{\delta_j^0}{k}k\theta_jv_j^0,(C_0\cap F)^\circ\right)^2\cr
&\leqslant&\frac{1}{k}\sum_{j\in J_0}\delta_j^0d\left(k\theta_jv_j^0,(C_0\cap F)^\circ\right)^2\cr
&\leqslant&\frac{1}{k}\sum_{j\in J_0}\delta_j^0d\left(k\theta_jv_j^0,\left[-\Vert P_F u_j\Vert_2 v_j^0,\Vert P_F u_j\Vert_2 v_j^0\right]\right)^2\cr
&=&\sum_{j\in J_0}\delta_j^0d\left(\sqrt{k}\theta_jv_j^0,\left[-\frac{\Vert P_F u_j\Vert_2}{\sqrt{k}} v_j^0,\frac{\Vert P_F u_j\Vert_2}{\sqrt{k}} v_j^0\right]\right)^2,
\end{eqnarray*}
and then setting, for every $j\in J_0$,
$$
f_j(t)=e^{-\pi d\left(\sqrt{k}tv_j^0,\left[-\frac{\Vert P_F u_j\Vert_2}{\sqrt{k}} v_j^0,\frac{\Vert P_F u_j\Vert_2}{\sqrt{k}} v_j^0\right]\right)^2},\quad\forall t\in\R
$$
we have that, for any $x\in F$ and any $\{\theta_j\}_{j\in J_0}\subseteq\R$ such that $\displaystyle{x=\sum_{j\in J_0}\delta_j^0\theta_jv_j^0}$,
$$
\prod_{j\in J_0} f_j^{\delta_j^0}(\theta_j)\leqslant e^{-\pi d(x,(K\cap F)^\circ)^2}.
$$
Therefore, by the reverse Brascamp-Lieb inequality,
\begin{eqnarray*}
\mathcal{W}((K\cap F)^\circ)&=&\int_{F}e^{-\pi d(x,(K\cap F)^\circ)^2}dx\geqslant\prod_{j\in J_0}\left(\int_{\R}f_j(t)dt\right)^{\delta_j^0}\cr
&=&\prod_{j\in J_0}\left(\frac{1}{\sqrt{k}}\int_{\R}e^{-\pi d\left(tv_j^0,\left[-\frac{\Vert P_F u_j\Vert_2}{\sqrt{k}} v_j^0,\frac{\Vert P_F u_j\Vert_2}{\sqrt{k}} v_j^0\right]\right)^2}dt\right)^{\delta_j^0}\cr
&=&\frac{1}{k^{k/2}}\prod_{j\in J_0}\mathcal{W}\left(\left[-\frac{\Vert P_F u_j\Vert_2}{\sqrt{k}} v_j^0,\frac{\Vert P_F u_j\Vert_2}{\sqrt{k}} v_j^0\right]\right)^{\delta_j^0}\cr
&=&\frac{1}{k^{k/2}}\prod_{j\in J_0}\left(1+\frac{2\Vert P_F u_j\Vert_2}{\sqrt{k}}\right)^{\delta_j^0}\geqslant \frac{1}{k^{k/2}}.
\end{eqnarray*}
\end{proof}
\section{Sections of convex bodies in minimal surface area position}

In this section we are going to prove Theorem \ref{thm:SurfaceArea}. Let us start assuming that $K$ is a centrally symmetric polytope in minimal surface area position and $F\in G_{n,k}$. By an approximation argument the inequalities obtained will also be true for any centrally symmetric convex body in minimal surface area position. We will follow the notation introduced in Section \ref{sec:GeneralSetting}.

Let, for every $j\in J_0$, $f_j:\R\to[0,\infty)$ be the function
$$
f_j(t)=e^{-\pi d(tv_j^0,P_{\langle v_j^0\rangle}(K\cap F))^2}\quad\forall t\in\R,
$$
where $\langle v_j^0\rangle$ denotes the $1$-dimensional subspace spanned by $v_j^0$. Notice that, for every $j\in J_0$,
$$
\int_{\R}f_j(t)dt=\int_{\R}e^{-\pi d(tv_j^0,P_{\langle v_j^0\rangle}(K\cap F))^2}dt=\mathcal{W}(P_{\langle v_j^0\rangle}(K\cap F)),
$$
and, since for every $j\in J_0$ we have that
$$
P_{\langle v_j^0\rangle}(K\cap F)\subseteq\left[-t_jh_K(u_j),t_jh_K(u_j)\right]v_j^0,
$$
we have that, for every $j\in J_0$,
$$
\int_{\R}f_j(t)dt\leqslant \mathcal{W}\left(\left[-t_jh_K(u_j),t_jh_K(u_j)\right]v_j\right)=\left(1+2t_jh_K(u_j)\right).
$$
Therefore, by the Brascamp-Lieb inequality
\begin{eqnarray*}
\int_{F}e^{-\pi\sum_{j\in J_0}\delta_j^0d\left(\langle x,v_j^0\rangle  v_j^0,P_{\langle v_j^0\rangle}(C_0\cap F) \right)^2}&=&\int_{F}\prod_{j\in J_0}f_j^{\delta_j^0}(\langle x,v_j^0\rangle) dx\cr
\leqslant\prod_{j\in J_0}\left(\int_{\R}f_j(t)dt\right)^{\delta_j^0}
&\leqslant&\prod_{j\in J_0}\left(1+2t_jh_K(u_j)\right)^{\delta_j^0}.
\end{eqnarray*}
By the arithmetic-geometric mean inequality we have
\begin{eqnarray*}
\prod_{j\in J_0}\left(1+2t_jh_K(u_j)\right)^\frac{\delta_j^0}{k}&\leqslant&\sum_{j\in J_0}\frac{\delta_j^0}{k}\left(1+2t_jh_K(u_j)\right)\cr
&=&1+\frac{2}{k}\sum_{j\in J_0}\frac{n|F_j|\Vert P_F u_j\Vert_2h_K(u_j)}{\partial(K)}\cr
&\leqslant&1+\frac{2n}{k\partial(K)}\sum_{j\in J_0}|F_j|h_K(u_j)\cr
&\leqslant&1+\frac{2n}{k\partial(K)}\sum_{j=1}^m|F_j|h_K(u_j)\cr
&=&1+2\frac{n^2|K|}{k\partial(K)},
\end{eqnarray*}
and then,
$$
\prod_{j\in J_0}\left(1+2t_jh_K(u_j)\right)^{\delta_j^0}\leqslant\left(1+2\frac{n^2|K|}{k\partial(K)}\right)^k=\mathcal{W}\left(\frac{n^2|K|}{k\partial(K)}B_\infty^k\right).
$$
On the other hand, for every $x_0\in K\cap F$ we have that, for every $x\in F$ and every $j\in J_0$,
$$
d\left(\langle x,v_j^0\rangle  v_j^0,P_{\langle v_j^0\rangle}(K\cap F) \right)^2\leqslant d\left(\langle x,v_j^0\rangle  v_j^0,\langle x_0,v_j^0\rangle v_j^0\right)^2=\langle x-x_0,v_j\rangle^2.
$$
Thus, for every $x_0\in C_0\cap F$ and every $x\in F$,
$$
\sum_{j\in J_0}\delta_j^0 d\left(\langle x,v_j^0\rangle  v_j^0,P_{\langle v_j^0\rangle}(K\cap F) \right)^2\leqslant\sum_{j\in J_0}\delta_j^0\langle x-x_0,v_j^0\rangle^2=|x-x_0|^2,
$$
and hence, for every $x\in F$,
$$
\sum_{j\in J_0}\delta_j^0 d\left(\langle x,v_j^0\rangle  v_j^0,P_{\langle v_j^0\rangle}(K\cap F) \right)^2\leqslant d(x,K\cap F)^2.
$$
Consequently,
\begin{eqnarray*}
\mathcal{W}(K\cap F)&=&\int_{F}e^{-\pi d(x,K\cap F)^2}dx\leqslant\int_{F}e^{-\pi\sum_{j\in J_0}\delta_j^0 d\left(\langle x,v_j^0\rangle  v_j^0,P_{\langle v_j^0\rangle}(K\cap F) \right)^2}dx\cr
&\leqslant&\mathcal{W}\left(\frac{n^2|K|}{k\partial(K)}B_\infty^k\right),
\end{eqnarray*}
which proves (i).

Since for every $\lambda\geqslant0$ we have that $\lambda K$ is in minimal surface area position, we have by (i) that, for every $\lambda\geqslant0$,
$$
\mathcal{W}(\lambda(K\cap F))\leqslant\mathcal{W}\left(\lambda\frac{n^2}{k}\frac{|K|}{\partial(K)}B_\infty^k\right).
$$
Therefore, as explained in Section \ref{subsec:Wills}, we obtain that $V_1(K\cap F)\leqslant V_1\left(\lambda\frac{n^2}{k}\frac{|K|}{\partial(K)}B_\infty^k\right)$ and $V_n((K\cap F)\leqslant V_n\left(\lambda\frac{n^2}{k}\frac{|K|}{\partial(K)}B_\infty^k\right)$, which is equivalent to (ii) and (iii).

Now, observe that for every $x\in F$ we have that
\begin{eqnarray*}
h_{K\cap F}(x)&=&\inf\left\{\sum_{j\in J_0}|\alpha_j|\,:\,x=\sum_{j\in J_0}\frac{\alpha_j}{h_K(u_j)}P_Fu_j\right\}\cr
&=&\inf\left\{\sum_{j\in J_0}|\beta_j|t_jh_K(u_j)\,:\,x=\sum_{j\in J_0}\beta_jv_j^0\right\},
\end{eqnarray*}
where we have set, for every $j\in J_0$, $t_j=\frac{1}{\Vert P_H u_j\Vert_2}=\left(\frac{c_j}{\delta_j^0}\right)^\frac{1}{2}$. For every $j\in J_0$, we define
$$
f_j(t):=e^{-|t|t_jh_K(u_j)},\ \ t\in\R.
$$
Then, if $\displaystyle{x=\sum_{j\in J_0}\delta_j^0\theta_jv_j^0}$ for some $\{\theta_j\}_{j\in J_0}\subseteq \R$, we have
$$
\prod_{j\in J_0} f_j^{\delta_j^0}(\theta_j)=e^{-\sum_{j\in J_0}\delta_j|\theta_j|t_jh_K(u_j)}\leqslant e^{-h_{K\cap F}(x)}.
$$
Therefore, by the reverse Brascamp-Lieb inequality,
\begin{eqnarray*}
k!|(K\cap F)^\circ|&=&\int_{F}e^{-h_{K\cap F}(x)}dx\geqslant\prod_{j\in J}\left(\int_{\R}e^{-|t|t_jh_K(u_j)}dt\right)^{\delta_j}\cr
&=&\frac{2^k}{\prod_{j\in J_0}(t_jh_K(u_j))^{\delta_j}}.
\end{eqnarray*}
By the arithmetic-geometric mean inequality,
\begin{eqnarray*}
\prod_{j\in J_0}\left(t_jh_K(u_j)\right)^\frac{{\delta_j^0}}{k}&\leqslant&\sum_{j\in J_0}\frac{\delta_j^0}{k}t_jh_K(u_j)=\sum_{j\in J_0}\frac{n|F_j|\Vert P_Fu_j\Vert_2 h_K(u_j)}{k\partial (K)}\cr
&\leqslant&\frac{n}{k\partial(K)}\sum_{j=1}^m|F_j| h_K(u_j)=\frac{n^2|K|}{k\partial(K)}.
\end{eqnarray*}
Taking into account that $|(B_\infty^k)^\circ|=|B_1^k|=\frac{2^k}{k!}$, we obtain
$$
|(K\cap F)^\circ|^{1/k}\geqslant\frac{k|\partial(K)}{n^2|K|}|(B_\infty^k)^\circ|^{1/k},
$$
which gives us (iv).

Finally, observe that for every $t\geqslant 0$
$$
t(K\cap F)=\{x\in F\,:\,|\langle x,v_j^0\rangle|\leqslant tt_jh_K(u_j),\forall\, j\in J_0\}.
$$
By the Brascamp-Lieb inequality he have that
\begin{eqnarray*}
\gamma_k(t(K\cap F))&=&\int_{\R^n}\prod_{j\in J_0}\chi_{[-tt_jh_K(u_j),tt_jh_K(u_j)]}(\langle x, v_j^0\rangle)\frac{e^{-\sum_{j\in J_0}\frac{\delta_j^0\langle x,v_j^0\rangle^2}{2}}}{(2\pi)^{k/2}}dx\cr
&\leqslant&\prod_{j\in J_0}\left(\int_{-tt_jh_K(u_j)}^{tt_jh_K(u_j)}\frac{e^{-\frac{t^2}{2}}}{\sqrt{2\pi}}\right)^{\delta_j^0}\cr
&=&\prod_{j\in J_0}\gamma_1\left(\left[-tt_jh_K(u_j),tt_jh_K(u_j)\right]\right)^{\delta_j^0}.
\end{eqnarray*}
Since the $1$-dimensional Gaussian measure is log-concave, we have that
\begin{eqnarray*}
\prod_{j\in J_0}\gamma_1\left(\left[-tt_jh_K(u_j),tt_jh_K(u_j)\right]\right)^{\delta_j^0}&\leqslant&\gamma_1\left(\left(\sum_{j\in J_0}\frac{\delta_j^0tt_jh_K(u_j)}{k}\right)[-e_1,e_1]\right)^k\cr
&=&\gamma_k\left(t\left(\sum_{j\in J_0}\frac{\delta_j^0t_jh_K(u_j)}{k}\right)B_\infty^k\right).
\end{eqnarray*}
Since
\begin{eqnarray*}
\sum_{j\in J_0}\frac{\delta_j^0t_jh_K(u_j)}{k}&=&\sum_{j\in J_0}\frac{n|F_j|\Vert P_Fu_j\Vert_2h_K(u_j)}{k\partial(K)}\leqslant\sum_{j\in J_0}\frac{n|F_j|h_K(u_j)}{k\partial(K)}\cr
&\leqslant&\sum_{j=1}^m\frac{n|F_j|h_K(u_j)}{k\partial(K)}=\frac{n^2}{k}\frac{|K|}{\partial(K)},
\end{eqnarray*}
we have that, for any $t\geq0$,
$$
\gamma_k(t(K\cap F))\leqslant\gamma_k\left(t\frac{n^2}{k}\frac{|K|}{\partial(K)}B_\infty^k\right).
$$
Therefore,
$$
w((K\cap F)^\circ)\geqslant w\left(\left(\frac{n^2}{k}\frac{|K|}{\partial(K)}B_\infty^k\right)^\circ\right)=\frac{k}{n^2}\frac{\partial(K)}{|K|}w((B_\infty^k)^\circ),
$$
and we obtain (v).

Let us now assume that $K$ is a (not necessarily symmetric) polytope in minimal surface area position and $F\in G_{n,k}$. Again, by approximation the inequalities obtained will be true for any convex body. Following the notation in Section \ref{sec:GeneralSetting} we have that, for any $x\in F$,
\begin{eqnarray*}
\|x\|_{\Pi^{*}K\cap F}&=&\frac{1}{2}\sum_{j=1}^m |F_j||\langle x,u_j\rangle|=\frac{1}{2}\sum_{j\in J_0} |F_j|\|P_Fu_j\|_2\left|\left\langle x, v_j^0\right\rangle\right|\cr
&=&\sum_{j\in J_0} \frac{\partial(K)\delta_j^0t_j}{2n}\left|\left\langle x, v_j^0\right\rangle\right|.
\end{eqnarray*}
Therefore, by the Brascamp-Lieb inequality
\begin{eqnarray*}
k!|\Pi^*K\cap F|&=&\int_{F}e^{-\Vert x\Vert_{\Pi^*K\cap H}dx}=\int_{F}e^{-\sum_{j\in J_0} \frac{\partial(K)\delta_j^0t_j}{2n}\left|\left\langle x, v_j^0\right\rangle\right|}dx\cr
&=&\int_F\prod_{j\in J_0}\left(e^{-\frac{\partial(K)t_j}{2n}\left|\left\langle x, v_j^0\right\rangle\right|}\right)^{\delta_j^0}dx\cr
&\leqslant&\prod_{j\in J_0}\left(\int_\R e^{-\frac{\partial(K)t_j}{2n}t}dt\right)^{\delta_j^0}=\prod_{j\in J_0}\left(\frac{4n}{\partial(K)t_j}\right)^{\delta_j^0}\cr
&=&\left(\frac{4n}{\partial(K)}\right)^k\prod_{j\in J_0}\left(\Vert P_F u_j\Vert\right)^{\delta_j^0}\leqslant\left(\frac{4n}{\partial(K)}\right)^k.
\end{eqnarray*}
On the other hand, we have that $P_F\Pi K=(\Pi^*K\cap F)^\circ$ and then, for every $x\in F$
$$
\Vert x\Vert_{P_F\Pi K}=\inf\left\{\max_{j\in J_0}|\tau_j|\,:\, x=\sum_{j\in J_0} \frac{\partial(K)\delta_j^0t_j\tau_j}{2n}v_j^{0}\right\}.
$$
Therefore, taking for every $j\in J_0$ the function $h_j=\chi_{\left[-\frac{\partial(K)t_j}{2n},\frac{\partial(K)t_j}{2n}\right]}$
we have that any decomposition of $x$ of the form $x=\sum_{j\in J_0}\delta_j^0\theta_jv_j^0$ with $|\theta_j|\leqslant \frac{\partial(K)t_j}{2n}$ gives a decomposition of $x$ of the form
$$
x=\sum_{j\in J_0}\frac{\partial(K)\delta_j^0t_j\tau_j}{2n}v_j^{0}\quad\textrm{with}\quad \tau_j=\frac{2n\theta_j}{\partial(K)t_j},
$$
and then $\max_{j\in J_0}|\tau_j|\leqslant1$. Therefore, the function $h=\chi_{P_F\Pi K}$ has the property that, for every $\{\theta_j\}_{j\in J_0}\subseteq\R$,
$$
h\left(\sum_{j\in J_0}\delta_j^0\theta_jv_j^0\right)\geqslant \prod_{j\in J_0}h_j^{\delta_j^0}(\theta_j),
$$
and hence, by the reverse Brascamp-Lieb inequality
\begin{eqnarray*}
|P_F\Pi K|&=&\int_{F}h(x)dx\geqslant\prod_{j\in J_0}\left(\int_{\R}h_j(t)dt\right)^{\delta_j^0}=\prod_{j\in J_0}\left(\frac{\partial(K)t_j}{n}\right)^{\delta_j^0}\cr
&=&\left(\frac{\partial(K)}{n}\right)^{k}\prod_{j\in J_0}\left(\frac{1}{\Vert P_Fu_j\Vert}\right)^{\delta_j^0}
\geqslant\left(\frac{\partial(K)}{n}\right)^{k}.
\end{eqnarray*}


\begin{thebibliography}{99}

\bibitem{A10} D.\ Alonso-Guti\'errez, {\it Volume estimates for $L_p$-zonotopes and best best constants in Brascamp-Lieb inequalities.}
Mathematika {\bf 56} (2010) 45--60.

\bibitem{AHY20} D.\ Alonso-Guti\'errez, M.~A.\ Hern\'andez Cifre and J.\ Yepes Nicol\'as, {\it Further inequalities for the (generalized) Wills functional.}
To appear in Commun. Contemp. Math. arXiv:1912.07993

\bibitem{B89} K.~M.\ Ball, {\it Volumes of sections of cubes and related problems.} Lecture Notes in Mathematics {\bf 1376}, Springer,
Berlin (1989), 251--260.

\bibitem{B91} K.~M.\ Ball, {\it Volume ratios and a reverse isoperimetric inequality.}
J. London Math. Soc. {\bf 44} (2), (1991) 351--359.

\bibitem{B92} K.~M.\ Ball, {\it Ellipsoids of maximal volume in convex bodies.}
 Geom. Dedicata {\bf 41} (1992), 241--250.

\bibitem{Ba98} F.\ Barthe, {\it An extremal property of the mean width of the simplex.}
Math. Ann. {\bf 310} (1998), 685–-693.

\bibitem{BGMN2005} F.\ Barthe, O.\ Gu\'edon,  S.\ Mendelson and A.\ Naor, {\it A probabilistic epproach to the geometry of the $\ell_p^N$-ball}
The Annals of Probability {\bf 33} (2) (2005), 480--513.

\bibitem{BGVV14} S.\ Brazitikos, A.\ Giannopoulos, P.\ Valettas and B.-H.\ Vritsiou, {\it Geometry of isotropic convex bodies.}
Mathematical Surveys and Monographs {\bf 196}, AMS, Providence Rhode, Island, 2014.

\bibitem{D16} H.\ Dirksen, {\it Sections of simplices and cylinders-Volume formulas and estimates.}
PhD dissertation (2015), Christian-Albrechts-Universit\"at zu Kiel.

\bibitem{GP99} A.\ Giannopoulos and M.\ Papadimitrakis, {\it Isoropic surface area measures.}
Mathematika {\bf 46}, (1999), 1--13.

\bibitem{GMR00} A.\ Giannopoulos, V.~D.\ Milman and M.\ Rudelson, {\it Convex bodies with minimal mean width.}
Geometric Aspects of Functional Analysis. Lecture Notes in Mathematics {\bf 1745}, (2000), 81--93.

\bibitem{GMR88} Y.\ Gordon, M.\ Meyer and S. Reisner, {\it Zonoids with minimal volume produt: A new proof.}
Proceddings of the AMS {\bf 104} (1988), 273--276.

\bibitem{Gr07} P.~M.\ Gruber, {\it Convex and Discrete Geometry.}
 (2007) Springer-Verlag Berlin Heidelberg.

\bibitem{Ha75} H.\ Hadwiger, {\it Das Wills'sche Funktional.} Monatsh. Math.
{\bf 79} (1975), 213--221.

\bibitem{J48} F.\ John, {\it Extremum problems with inequalities as subsidiary conditions.} Courant Anniversary Volume, Inter-
science, New York (1948), 187--204.

\bibitem{Ma20}E.\ Markessinis, {\it Central sections of a convex body with ellipsoid of maximal volume $B_2^n$.}
Mathematische Annalen {\bf 378}, (2020) 233--241.

\bibitem{MPS12} E. Markessinis, G. Paouris and Ch. Saroglou, {\it Comparing the M-position with some classical positions of convex bodies.}
Math. Proc. Cambridge Philos. Soc. {\bf 152}, (2012), 131--152.

\bibitem{Mc75} P.\ McMullen, Non-linear angle-sum relations for
polyhedral cones and polytopes, {\it Math. Proc. Cambridge Philos.
Soc.} {\bf 78} (1975), 247--261.

\bibitem{Mc91} P.\ McMullen, Inequalities between intrinsic volumes, {\it
Monatsh. Math.} {\bf 111} (1) (1991), 47--53.

\bibitem{P61} C.~M. Petty, {\it Surface area of a convex body under affine transformations.}
Proceedings of the AMS {\bf 12} (1961), 824--828.

\bibitem{R86} S.\ Reisner, {\it Zonoids with minimal volume produt.}
Math. Z. {\bf 192}, (1986), 339--346.

\bibitem{SY93} J.~R.\ Sangwine-Yager, {\it Mixed volumes.} In: Handbook of
Convex Geometry (P. M. Gruber and J. M. Wills eds.),
North-Holland, Amsterdam, 1993, 43--71.

\bibitem{ScS95} G.\ Schechtman and M.\ Schmuckenschl\"ager, {\it A concentration inequality for harmonic measures
on the sphere.} In Geometric aspects of functional analysis (Israel, 1992–1994), {\bf 77}
in Oper. Theory Adv. Appl. Basel: Birkh\"auser (1995), 255–-273.

\bibitem{Sc99} M.\ Schmuckensl\"ager, {\it An extremal property of the regular simplex.}
Convex Geometric Analysis (Berkeley, CA, 1996), MSRI Publications Cambridge Univ. Press,
Cambridge, {\bf 34} (1999) 199--202.

\bibitem{Sch} R.\ Schneider, {\it Convex bodies: The Brunn-Minkowski
theory}, 2nd expanded ed. Encyclopedia of Mathematics and its Applications
151, Cambridge, Cambridge University Press, 2014.

\bibitem{We96} S.\ Webb, {\it Central Slices of the Regular Simplex.} Geom. Dedicata {\bf 61} (1996), 19--28.

\bibitem{W73} J.~M.\ Wills, {\it Zur Gitterpunktanzahl konvexer Mengen.} Elem. Math. {\bf 28} (1973), 57--63.

\bibitem{W75} J.~M.\ Wills, {\it Nullstellenverteilung zweier konvexgeometrischer
Polynome.} Beitr\"age Algebra Geom. {\bf 29} (1989), 51--59.
\end{thebibliography}

\end{document}